\documentclass[12pt,reqno]{amsart}

\usepackage{amssymb,amsthm,amsmath}
\usepackage{natbib}

\usepackage{graphicx,psfrag,epsf}
\usepackage{enumerate}
\usepackage{url} 
\usepackage{tikz}
\usepackage[colorinlistoftodos, textwidth=2cm, shadow]{todonotes}

\newcommand{\abs}[1]{\left| #1 \right|}
\newcommand{\norm}[1]{\left\|{#1}\right\|}
\newcommand{\R}{\mathbb{R}}
\newcommand{\N}{\mathbb{N}}
\newcommand{\E}[1]{\mathbf{E}\left[#1\right]}
\renewcommand{\P}{\mathbb{P}}

\theoremstyle{plain}
\newtheorem{thm}{Theorem}[section]

\theoremstyle{definition}
\newtheorem{rem}[thm]{Remark}
\newtheorem{lem}[thm]{Lemma}

\allowdisplaybreaks

\begin{document}

\def\spacingset#1{\renewcommand{\baselinestretch}%
{#1}\small\normalsize} \spacingset{1}

\title[FDR-Control in Segmentation]{ FDR-Control in Multiscale Change-point Segmentation}

\author[H. Li, A. Munk, and H. Sieling]{Housen Li,  
   Axel Munk\hspace{.2cm}\\
    Institute for Mathematical Stochastics, University of G\"{o}ttingen \\
    and Max Planck Institute for Biophysical Chemistry \\
G\"{o}ttingen, Germany\\
    and \\
    Hannes Sieling \\
    Institute for Mathematical Stochastics, University of G\"{o}ttingen\\
    G\"{o}ttingen, Germany}

\maketitle

\begin{abstract}
Fast multiple change-point segmentation methods, which additionally provide faithful statistical statements on the number, {locations} and sizes of the segments, have recently received great attention. In this paper, we propose a multiscale segmentation method, FDRSeg, which controls the false discovery rate (FDR) in the sense that  the number of false jumps is bounded linearly by  the number of true jumps. In this way, it adapts the detection power to the number of true jumps.   
We {prove} a non-asymptotic upper bound for its FDR in a Gaussian setting, which allows to calibrate the only parameter of FDRSeg properly.  Change-point locations, as well as the signal, are shown to be estimated in a uniform sense at optimal {minimax} convergence rates up to a log-factor.  {The latter is w.r.t. $L^p$-risk, $p \ge 1$, over classes of step functions with bounded jump sizes and either bounded, or possibly increasing, number of change-points.}
{FDRSeg can be efficiently computed by an accelerated dynamic program; its computational complexity is shown to be linear in the number of observations when there are many change-points. } 
The performance of the proposed method is examined by comparisons with some state of the art methods on both simulated and real datasets. {An R-package is available online. }
\end{abstract}


\spacingset{1.15}

\section{Introduction}\label{sec:intro}
To keep the presentation simple, we assume that observations are given by the regression model
\begin{equation}\label{model_1}
 Y_i = \mu\left(\frac{i}{n}\right) + \sigma \varepsilon_i, \quad i=0,\,1,\,\ldots,\,n-1,
\end{equation}
where $\varepsilon_0,\ldots,\varepsilon_{n-1}$ are independent standard normally distributed, and $\sigma > 0$. The mean-value function $\mu$ is assumed to be right-continuous and piecewise constant with $K+1$ segments {$I_k=[\tau_k,\tau_{k+1}) \subset [0,1)$}, i.e.
\begin{equation}\label{eq:true_signal}
 \mu = \sum_{k=0}^{K} c_k \mathbf{1}_{[\tau_k,\tau_{k+1})}.
\end{equation}
Here {the number of change-points $K$ is unknown, as well as the change-points $\tau_k$, $0 < \tau_1 < \ldots < \tau_K < 1$}, with the convention that $\tau_0 :=0$ and $\tau_{K+1}:=1$. The (unknown) value of  $\mu$ on the $k$-th segment $I_k$ is denoted by $c_k$ {and we assume $c_k \neq c_{k+1}, k= 0,1,\ldots, K-1$ for identifiability of $\mu$}. We stress, however, that much of our subsequent methodology and analysis can be extended to other models, e.g. {for nonequidistant sampling points}, when the observations come from an exponential family or more generally, errors obey certain moment conditions, {and} to dependent data. {The latter case will be illustrated in Section~\ref{subsec:ion:channel} for the segmentation of ion channel recordings. }

Estimation of $\mu$ and its change-points in this seemingly simple model~\eqref{model_1} (and variations thereof) has a long history in statistical research~(see e.g. \citep{CarMueSie94,CsoHor97,Sie13,FriMunSie14} for a survey). It has recently gained renewed interest from two perspectives, in particular. Firstly, large scale applications such as from finance~(see e.g. \citep{IncTia94,BayPer98,LavTey07,Spo09,DavHoeKra12}), signal processing~(see e.g. \citep{HarLev08,BlyBunMeiMul12,Hot12}) or genetic engineering~(see e.g. \citep{BraMueMue00,OlsVenLucWig04,ZhaSie07,ZhSie12,Jen10,Sie13}) call for change-point segmentation methods which are computationally fast, say almost linear in the number of observations. Secondly, besides of a mere segmentation of the data into pieces of constancy certain evidence on {the number, locations and heights of these pieces which come with this segmentation is demanded.}  

Many {state of the art segmentation} methods {which aim to meet the latter two goals} are based on minimizing a penalized cost functional among different number of change-points $K$ and locations of change-points $\tau_k$. For a cost function $C$, which serves as goodness-of-fit measure of a constant function on an interval, and a penalty against over-fitting $f(K)$ these approaches search for a solution of the global optimization problem
\begin{equation}\label{def:pen_cost}
\min_{\mu} \sum_{k=0}^{K} C(Y_{\lceil n\tau_{k}\rceil}, \ldots, Y_{\lceil n\tau_{k+1}\rceil-1}; c_k) + \gamma_n f(K).
\end{equation}
Fast and exact algorithms for this kind of methods {employ}  dynamic programming such as the optimal partitioning method \citep{Jac05} and the Potts estimate~\citep{BoyKemLieMunWit09, SWD14}, who advocate the sparsest subset selection penalty 
\begin{equation}\label{eq:l0:penalty}
f(K)=l_0({\mu}) = K.
\end{equation}
For more general $f$, see e.g. the segment neighbor method \citep{AugLaw89} or \citep{FriKemLieWin08}. More recently, Killick et al. \citep{KillFeaEck12} introduced a pruned dynamic program (PELT) with expected linear complexity mainly for $f(K)= K$ {and Du et al. \citep{DKK15} used dynamic programming to compute the marginal MLE in a Bayesian framework.}
From a computational point of view, approaches of type \eqref{def:pen_cost} seem therefore  beneficial.  Nevertheless, the
choice of {$f$ and its associated} balancing parameter $\gamma_n=\gamma_{n}(Y)$  in~\eqref{def:pen_cost} is subtle. Birg\`{e} and Massart \citep{Birge2006} offer examples and discussion of  this and other penalty choices, and Boysen et al. \citep{BoyKemLieMunWit09} provide {asymptotically} optimal choices of $\gamma_n$, as $n \to \infty$.  Zhang and Siegmund \citep{ZhaSie07, ZhSie12} proposed a penalty depending on $K$ and additionally on distances between consecutive change-points.

%

In contrast to solving the global optimization problem in~\eqref{def:pen_cost} another prominent class of methods is based on the idea to iteratively apply a local segmentation method to detect a single change-point. If such a change-point is detected on a segment, it is split into two parts and the same routine is applied to both new segments. The method stops if no further change-points are found. This approach, referred to as \emph{binary segmentation} (BS), is certainly among the most popular ones for change-point segmentation, in particular in the context of the analysis of copy number variation data and related biostatistical issues. It has already been suggested in \citep{ScoKno74} and more recently related methods have been proposed, such as \emph{circular binary segmentation} (CBS) \citep{OlsVenLucWig04, Ven07} and \emph{wild binary segmentation} (WBS) \citep{Fry14}. For these approaches, the to be specified parameter {among others} is the probability of including a false change-point in one iteration. Therefore, local error control can be provided, but the overall {uniform} control on the error to include or exclude wrong segments appears to be {often difficult} for these methods, as well. {A notable exception is~\citep[Theorems 3.2 and 3.3]{Fry14}, however, these bounds depend on constants which are difficult to specify. }

However, given the data at hand, {significant} conclusions  on the number, location and size of the change-point function are not an easy task for the above mentioned methods as these require uniform {finite sample error bounds, for all these quantities, simultaneously. }
A similar comment applies to other global segmentation methods which rely on an $l_1$ approximation of the {nonconvex} $l_0$ penalty in~\eqref{eq:l0:penalty} including lasso-type techniques possibly together with post filtering to further enhance sparseness, see e.g.~\citep{TibSauRosZhuKni04,FriHasHoeTib07,HarLev10}. 

Frick et al. \citep{FriMunSie14} suggest a hybrid method,  \emph{simultaneous multiscale change-point estimator} (SMUCE), {which tries to address both tasks {(computationally fast while still obeying finite sample uniform error control)} by minimizing the number of change-points under a local \emph{multiscale} side-constraint, see also \citep{BoyKemLieMunWit09,DavHoeKra12} for related estimators.} The side-constraint is based on a simultaneous multiple testing procedure on all scales (length of subsequent observations) which employs a scale calibrating penalty~\citep{DueSpok01}. It can be shown that for the resulting segmentation $\hat \mu$ the number of change-points is not overestimated at a pre-defined probability, $1-\alpha_S$ (i.e.  \emph{family-wise error rate}, FWER). This provides a direct statistical interpretation.
In fact, the error of including $j$ false positives provided by SMUCE has exponential decay, 
\begin{equation}\label{eq:overestimation:bound}
\P\{\hat K \ge K+j\} \le \alpha_S^{\lceil j/2 \rceil},\, j =1,2,\ldots 
\end{equation}
(see \citep{FriMunSie14}), which in particular controls the overestimation of the number of true {change-points $K$} ($j=1$ in~\eqref{eq:overestimation:bound})
\begin{equation}\label{eq:overestimation:bound:alpha}
\P\{\hat K > K \}\le \alpha_S.
\end{equation} 
Moreover, it can be shown that the method is able to detect the true {$K$} 
over a large range of scales with minimax detection power~\citep[Theorem 5]{FriMunSie14}.
%
%
However, according to~\eqref{eq:overestimation:bound:alpha}, in particular in situations with low signal to noise ratio (SNR) or with many change-points compared to the number of observations, this error control necessarily leads to a conservative estimate $\hat \mu$ of $\mu$ in~\eqref{eq:true_signal}, i.e. with fewer change-points than the true number $K$.
Therefore, in this paper we offer a strategy to overcome this drawback which might be beneficial also for other related methods. 
This is based on the control of the \emph{false discovery rate} (FDR) \citep{BenHoc95} instead of the FWER {control in~\eqref{eq:overestimation:bound:alpha}}. Despite of the huge literature about change-point {segmentation and detection}, there is only a small number of papers addressing the FDR issue in this context. {Early references include \citep{TibWan08} {which} proposed a multiple stage procedure, and gave empirical evidence for the FDR control, and 
\citep{Efron2011} {which} considered a local FDR based approach {for} the copy number {variation} analysis of multiple samples {in cancer genetics.} Recently, Hao et al.
\citep{HNZ13} proved the FDR control of the screening and ranking algorithm~\citep{NiZh12} for a restricted definition of FDR, and Cheng and Schwartzman \citep{ChSch15} provided an asymptotic control of FDR of a smoothing based approach. {For further discussion} see Section~\ref{subsec:multiplicity:fdr}.}

\subsection{FDRSeg}

In this work, we will present  \emph{FDRSeg}, which controls the \emph{FDR} of the whole segmentation. 
The significance statement given by the method is quite intuitive and also holds for a finite number of observations. This reveals the contribution of this work {as threefold}: First, the new method overcomes the conservative nature of SMUCE and variants (see~\citep{YSSamworth14}) while maintaining a solid statistical interpretation. In doing this, we provide a general framework how to combine FDR-control with global segmentation methods {in a multiscale fashion}, which is of interest by its own. 
Second, {various} {optimality} statements are provided, and all results hold in a non-asymptotic {manner}  uniformly over a large class of piecewise constant functions $\mu$ in model~\eqref{eq:true_signal}. {Third, FDRSeg is shown to be computable often in almost linear time. In summary, FDRSeg is a hybrid segmentation technique, combining statistical efficiency and fast computation while providing solutions with preassigned statistical accuracy.}

\begin{figure}[!h]
\centering
 \includegraphics[width= 0.98\columnwidth]{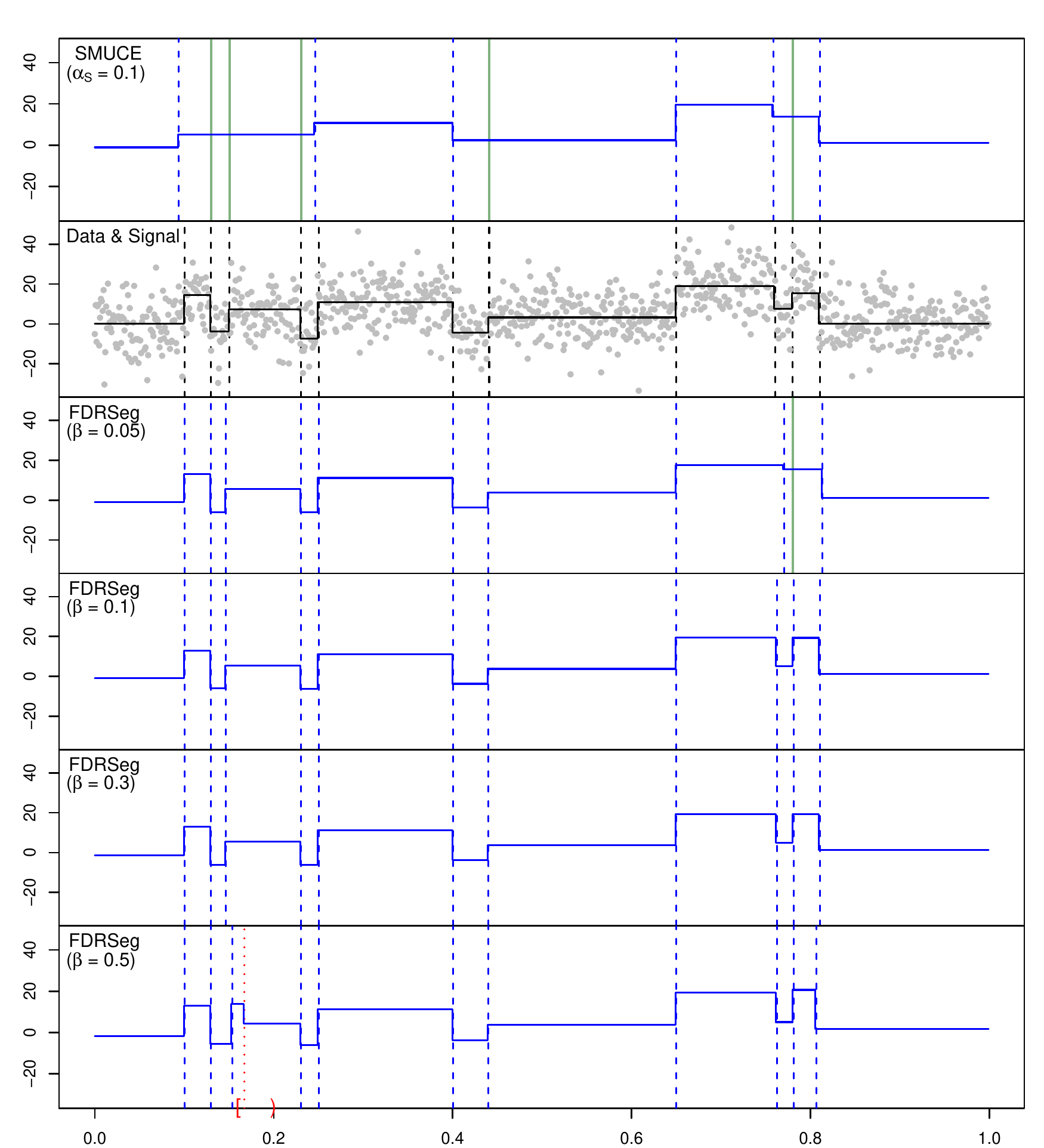} 
 \caption{Illustration of FDRSeg. The noisy data together with the true signal is shown in the second panel. Below, FDRSeg ($\beta = {0.05}$), FDRSeg ($\beta=0.1$), FDRSeg ($\beta = 0.3$), and FDRSeg ($\beta = 0.5$) are shown. As a comparison, SMUCE ($\alpha_S = 0.1$) is shown on the top. Each true discovery is indicated by a vertical blue dashed line and each false one by a vertical red dotted line and an associated interval defined in~\eqref{eq:interval_changepoint}. The vertical green lines indicate missed change-points.}
 \label{pic:intro}
\end{figure}

Before going into details, we illustrate our approach by the example in Figure~\ref{pic:intro}. We employed the \emph{blocks signal} \citep{DonJoh94} with Gaussian observations of standard deviation $\sigma = 10$ (with integrated SNR  $\int \abs{\mu(x)}dx/\sigma\approx 0.65$). 
Very naturally we declare such discoveries (estimated change-points) true if they are ``close'' (to be specified later) to true change-points. In this example  FDRSeg ($\beta = 0.1$) detects all the change-points correctly, while SMUCE $(\alpha_S = 0.1)$ finds only $6$ out of $11$, due to its requirement to control the  FWER in~\eqref{eq:overestimation:bound:alpha}. This remains valid until $\beta$ is increased to $\beta \approx 0.5$. For larger $\beta$ FDRSeg overestimates the number of change-points. {For example, if $\beta = 0.5$ it}  finds one additional false change-point (at $0.17$, marked by a vertical red line and an associated interval defined in~\eqref{eq:interval_changepoint}, in the bottom panel) besides all the true ones.  The proportion between false and all discoveries plus one (number of segments) is hence $1/(12+1)\approx 0.08 \ll 0.5$. Later we will show that FDRSeg is indeed able to control this proportion in expectation at {any} predefined level $\beta$ {uniformly over all possible change-point functions $\mu$}. For the other direction, the largest $\beta$ for {which FDRSeg}  underestimates the number of change-points is $0.07$ {(see the third panel {for $\beta = 0.05$}; the missing change-point is marked by a vertical green line).} That is, FDRSeg estimates the correct number of change-points for {the entire range of} $\beta \in (0.07, 0.50)$ {and hence appears to be remarkably stable in terms of the control parameter $\beta$.  This will be investigated more detailed later. }


\subsection{Multiplicity and FDR control}\label{subsec:multiplicity:fdr}
For our purpose it is helpful to interpret the ``detection part'' of the multiple change-point regression problem as a multiple testing problem. In the literature methods with this flavor often consider multiscale local likelihood tests. 
Whereas local tests for the presence of a change-point on small systems of sets (e.g. the dyadics) of the sampling points $\{0,1/n, \ldots, (n-1)/n\}$ can be efficiently computed they may have low detection power and highly redundant systems such as the system of all intervals have been suggested instead~\citep{SieYak00,DueSpok01,FriMunSie14}. {See}, however, \citep{Wal10,RivWal12} for {less redundant} but  still asymptotically efficient systems. 
It was pointed out in~\citep{SieZhaYak11} that  classical FDR for redundant systems might be misleading, because such local tests are highly correlated and consequently tests on nearby intervals likely reject/accept the null-hypothesis together, see also~\citep{BenYek01,GuoSar13} for a general discussion of this issue. Siegmund et al. \citep{SieZhaYak11} therefore suggest to test for constancy on subintervals and to group the nearby false (or true) rejections, and count them as a single discovery, which allows to control the FDR group-wise.  
In our approach, we circumvent this, but still are able to work with redundant systems, because instead we  perform a multiple test for the change-points directly,  i.e. we treat the multiple testing problem
\[
H_{i}: \frac{i}{n} \text{ is not a change-point,} \text{ v.s. } A_{i}:  \frac{i}{n} \text{ is a change-point,} \quad i = 0,\ldots, n-1. 
\]
It remains to define a true/false discovery. This is done by identifying a rejection as a true discovery if it is ``close'' to a true change-point. To be specific, let $\{\hat{\tau}_1, \ldots, \hat{\tau}_{\hat{K}}\}$ be rejections (i.e. estimated change-points), and $\hat K$ the estimated number of change-points. For each $i \in \{1,\ldots, \hat{K}\}$, we classify $\hat{\tau}_i$ as a \emph{true discovery} if there is a true change-point lying in 
\begin{equation}\label{eq:interval_changepoint}
\left[\frac{\lceil n(\hat{\tau}_{i-1} + \hat{\tau}_{i})/2\rceil}{n}, \frac{\lceil n(\hat{\tau}_{i} + \hat{\tau}_{i+1})/2\rceil}{n}\right)
\end{equation}
where $\hat{\tau}_0 := 0$ and $\hat{\tau}_{\hat{K}+1} := 1$; otherwise, it is a \emph{false discovery}, see again the bottom panel in Figure~\ref{pic:intro}. 
Similar to \citep{BenHoc95}, we then define the false discovery rate (FDR) by  
\begin{equation}\label{def:fdr}
\text{FDR} :=  \E{\frac{\text{FD}}{\hat{K}+1}},
\end{equation}
where FD is the number of false discoveries in the above sense. 
Note, that the above notion of true/false discoveries is \emph{well defined}: (a) every estimated change-point is either true or false, but not both; (b) corresponding to each true change-point there is at most one true discovery, because the intervals~\eqref{eq:interval_changepoint} are disjoint for different $i$. We stress that no additional assumption, such as the sparsity of change-points, the minimal length of segments, is needed for this definition.  It automatically adapts to the individual length of segments, in particular for the region of rapid changes, such as subgating characteristic of ion channel recordings~\citep{Hot12}.  To some extent it neglects the accuracy of jump locations, especially when the change-points are far apart located. In this sense, this definition primarily focuses  on the {correct} \emph{number} of change-points rather than the \emph{locations} {or the \emph{sizes} of the segments}.
In the following we will see however, that our method will also have a high accuracy in estimating the locations. To this end we will consider the following evaluation measure
\begin{equation}\label{eq:def:accLoc}
d(\mu,\hat\mu) := \max_{0 \le i \le K+1} \min_{0 \le j \le \hat{K}+1} \abs{\tau_i -\hat\tau_j},
\end{equation} 
for  $\mu = \sum_{i=0}^{K} \mathbf{1}_{[\tau_i,\tau_{i+1})} c_i$ and $\hat\mu = \sum_{j=0}^{\hat{K}} \mathbf{1}_{[\hat\tau_j,\hat\tau_{j+1})} \hat{c}_j$, {with the convention that $\tau_0 = \hat\tau_0 = 0$ and $\tau_{K+1} = \hat\tau_{\hat K +1} = 1$. } {In addition, we will examine the $L^p$-risk ($1 \le p < \infty$) of $\hat \mu$. }

\subsection{Plan of the paper}
The rest of the paper is organized as follows.  
In Section~\ref{sec:method}, we introduce the new segmentation method {FDRSeg} and show its FDR control. In Section~\ref{sec:risk:fdrseg} we prove a {finite sample} exponential deviation bound for the estimation {error} of the jump locations in~\eqref{eq:def:accLoc} (see Theorem~\ref{thm:accuracy:location}). From this we derive that the locations are estimated at the optimal sampling rate $O(1/n)$ up to a log-factor uniformly over a large class of sequences of step functions $\mu$ with possibly increasing number of change-points, minimal scale of order $\log (n)/n$, and non-vanishing minimal jump height. 
Further, for the estimate $\hat\mu$ we show that its $L^p$-risk ($1\le p < \infty$) is of order $(\log(n)/n)^{\min\{1/2, 1/p\}}$ (see Theorem~\ref{thm:convergence:rate:LpRisk}) in the class of step functions with minimal scale and jump bounded away from zero and bounded jump size. In Theorem~\ref{thm:convergence:rate:lower:bound} we prove a lower bound for the $L^p$-risk which reveals FDRSeg to be minimax optimal up to a log-factor in this class.  
 
In Section~\ref{sec:impl} we will develop a pruned dynamic program for {the computation of FDRSeg}. {It has linear {memory} complexity, and linear time complexity for signals with many change-points, in terms of the number of observations. } The accuracy and efficiency of FDRSeg is examined in Section~\ref{sec:sim_app} on both simulated and real datasets. {Compared to state of the art methods, FDRSeg shows a {high} power in detecting change-points and {high efficiency for} signal recovery on various scales, simultaneously. } {As demonstrated on ion channel recordings,  a modification to dependent data (D-FDRSeg) reveals {relevant} gating characteristics, {but avoids at the same hand spurious change-points which are misleadingly found without adaptation to the correlated noise}.}
The paper ends with a conclusion in Section~\ref{sec:conclusion}. 

 An implementation of FDRSeg is provided in R-package { ``FDRSeg''}, available from \url{http://www.stochastik.math.uni-goettingen.de/fdrs}.

\section{Method and FDR control}\label{sec:method}
Now we will give a formal definition of the FDRSeg. To simplify, we assume that the noise level $\sigma$ is known. For methods to estimate $\sigma^2$, see {\eqref{eq:noise:level:estimation} or e.g. \citep{rice1984,HKT90,DetMunWag98} among many others.} Assume that $Y=(Y_0, \ldots, Y_{n-1})$ is given by model~\eqref{model_1}. 
For an interval $I \subset [0,1)$ we consider the \emph{multiscale statistic} with scale calibration (motivated from~\citep{FriMunSie14})
 \begin{equation}\label{eq:test:stat:segment}
 T_I(Y, c) = \max_{ [i/n,j/n] \subset I }  \frac{\left| \sum_{l=i}^{j} (Y_l- c) \right|}{\sigma \sqrt{j-i+1}} - \text{pen} \left(\frac{j-i+1}{\#{I}}\right),
 \end{equation}
where $c$ is a real number, $\text{pen} (x) = \sqrt{2\log({e}/{x})}$ the penalty term for the scale and $\#{I}$ the number of {sampling points $i/n$} in $I$ (scale). The first term in~\eqref{eq:test:stat:segment} describes how well the data can be  \emph{locally}  described by the constant $c$ on the interval $[i/n,j/n] \subset I$, and the second term (so called scale calibration) is designed to balance the detection power among different scales (i.e. lengths of intervals), see~\citep{DueSpok01,FriMunSie14} for further details. Thus, $T_I(Y,c)$ examines the hypotheses that $\mu \equiv c$ on the interval $I$  simultaneously over all intervals $\subset I$, i.e. in particular on all scales of $I$. 

For  $\alpha \in (0,1)$, let us introduce \emph{local} quantiles $q_{\alpha}(m)$, $m = 1, \ldots, n$, by
\begin{equation}\label{def:qalpha}
q_{\alpha}(m) := \min\left\{q: \P\left\{T_I(\varepsilon, \bar{\varepsilon}_I) > q\right\} \le \alpha\right\}, 
\end{equation}
where $\varepsilon = (\varepsilon_0, \ldots, \varepsilon_{n-1})$ is standard normally distributed,  $\bar{\varepsilon}_I = \sum_{i/n\in I}\varepsilon_{i}/\#{I}$, and $I$ a fixed interval with $\#{I} = m$. Obviously, $q_{\alpha}(m)$ does not depend on the choice of $I$ if $\#{I}=m$, which justifies the definition~\eqref{def:qalpha}. 

\begin{rem}\label{rem:quantiles:uniform:bounded}
 As a direct consequence of \citep{DueSpok01} (see also~\citep{DueWal08,FriMunSie14}) the limit distribution of $T_I(\varepsilon, \bar{\varepsilon}_I)$ is finite almost surely and  is continuous \citep{DuePitZho06}, as  $\#{I} \to \infty$. 
{For every $\alpha \in (0,1)$,} the values $q_\alpha(m)$'s are therefore uniformly bounded for all $m$.  
In practice, $q_{\alpha}(m)$'s are obtained by Monte-Carlo simulations. Note, that this needs only to be done once and can be stored in a table, as it does not depend on the data nor the signal $\mu$. 
\end{rem}

For our purpose we have to introduce the set of step functions restricted to the multiscale side-constraint induced by~\eqref{eq:test:stat:segment} and~\eqref{def:qalpha} (for fixed $\alpha$)
\begin{equation}\label{side-constraint:fdrsmuce}
\mathcal{C}_k =  \left\{\mu =  \sum_{i=0}^{k} c_i \mathbf{1}_{I_i}: T_{I_i}(Y, c_i) - q_{\alpha}(\#{I_i}) \le 0\quad \forall i = 0, 1, \ldots, k\right\}.
\end{equation}
The estimated number of change-points $\hat K$ according to FDRSeg will then be given by
\begin{equation}\label{def:Khat}
\hat{K} := \min \left\{k:  \mathcal{C}_k \ne \emptyset \right\}.
\end{equation}
The  $\hat K$ will be always an integer between $0$ and  $n$, since  $\sum_{i=0}^{n-1}Y_i\mathbf{1}_[i/n, (i+1)/n) \in \mathcal{C}_{n-1}$.
The FDRSeg estimate $\hat \mu$ is  given by
\begin{equation}\label{def:fdr_smuce}
\hat \mu := \mathop{\arg\min}_{\mu \in \mathcal{C}_{\hat K}}  \sum_{i=0}^{n-1} \left(Y_i-\mu\left(\frac{i}{n}\right)\right)^2,
\end{equation}
that is, the constrained maximum likelihood estimator within $\mathcal{C}_{\hat K}$. 
The intuition behind is to search for the simplest step function (with complexity measured by number of change-points) which lies in the multiscale constraint in the form of~\eqref{side-constraint:fdrsmuce}. 

The main result of this section is that our estimator is able to control the FDR in the sense of~\eqref{def:fdr} by choosing the local levels $\alpha(m)$  for intervals of length $m$  in~\eqref{def:qalpha} properly. 

\begin{thm}\label{thm:mainthm}
Let $Y$ be observations from model~\eqref{model_1}, and $0 < \alpha < 1/3$. Then {FDRSeg in~\eqref{side-constraint:fdrsmuce}-\eqref{def:fdr_smuce} with $q_\alpha$ in~\eqref{def:qalpha} controls the FDR defined in~\eqref{def:fdr},}
 \begin{equation}\label{eq:thm:bnd}
  \mathrm{FDR}_{\hat\mu}(\alpha) \leq \frac{2\alpha}{1-\alpha}=:\beta.
 \end{equation}
\end{thm}
\begin{proof}
See Appendix~\ref{sec:proof_main_theorem}.
\end{proof}

\begin{rem}[Discussion of the bound]\label{rem:diss:bound}
Various simulation studies (not displayed) suggest even the bound $\mbox{FDR } \le \alpha$, improving~\eqref{eq:thm:bnd} by a factor of 2. Although we were not able to prove this, we stress that this might be useful for practical purpose to select and interpret $\alpha$. For example in Figure~\ref{fig:fdr:simulation:bound} we display results for the teeth signal (see Figure~\ref{fig:relation_choice_alpha_est}), where the FDR is estimated by the empirical mean of 1,000 repetitions with $n=600$. It shows that the bound~\eqref{eq:thm:bnd} (dashed line) is good when $\alpha$ is small, and gets worse as $\alpha$ increases. 
\end{rem}

\begin{figure}[t]
\centering
 \includegraphics[width=0.50\columnwidth]{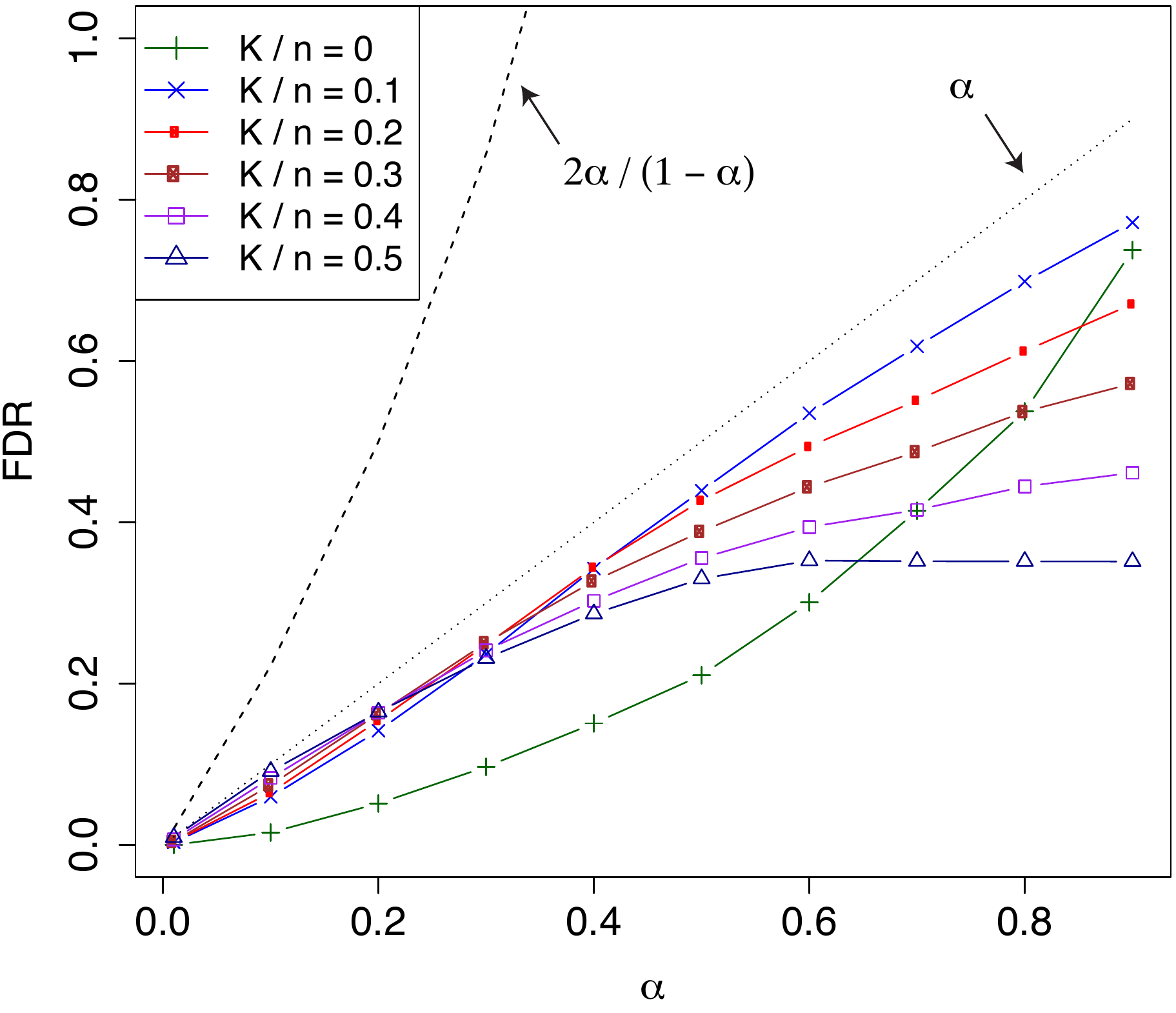}
 \caption{Simulation on the bound of FDR.  }
 \label{fig:fdr:simulation:bound}
\end{figure}

\begin{rem}[Choice of parameter for FDRSeg]{
Note that Theorem~\ref{thm:mainthm} provides a statistical guidance for the choice of the only parameter $\alpha$ for FDRSeg. To calibrate the method for given $\beta$, we simply rewrite~\eqref{eq:thm:bnd} into
\[
\alpha = \frac{\beta}{2+\beta},
\]
which is roughly, $\alpha = \beta/2$ for small $\beta$, see Figure~\ref{fig:fdr:alpha}. In practice, one could even use $\alpha = \beta$ as discussed in Remark~\ref{rem:diss:bound}. We further stress that FDRSeg is actually robust to the choice of $\beta$ (or $\alpha$), as we have already seen in Figure~\ref{pic:intro}.}
\end{rem}

\begin{figure}[!h]
\centering
 \includegraphics[width=0.48\columnwidth]{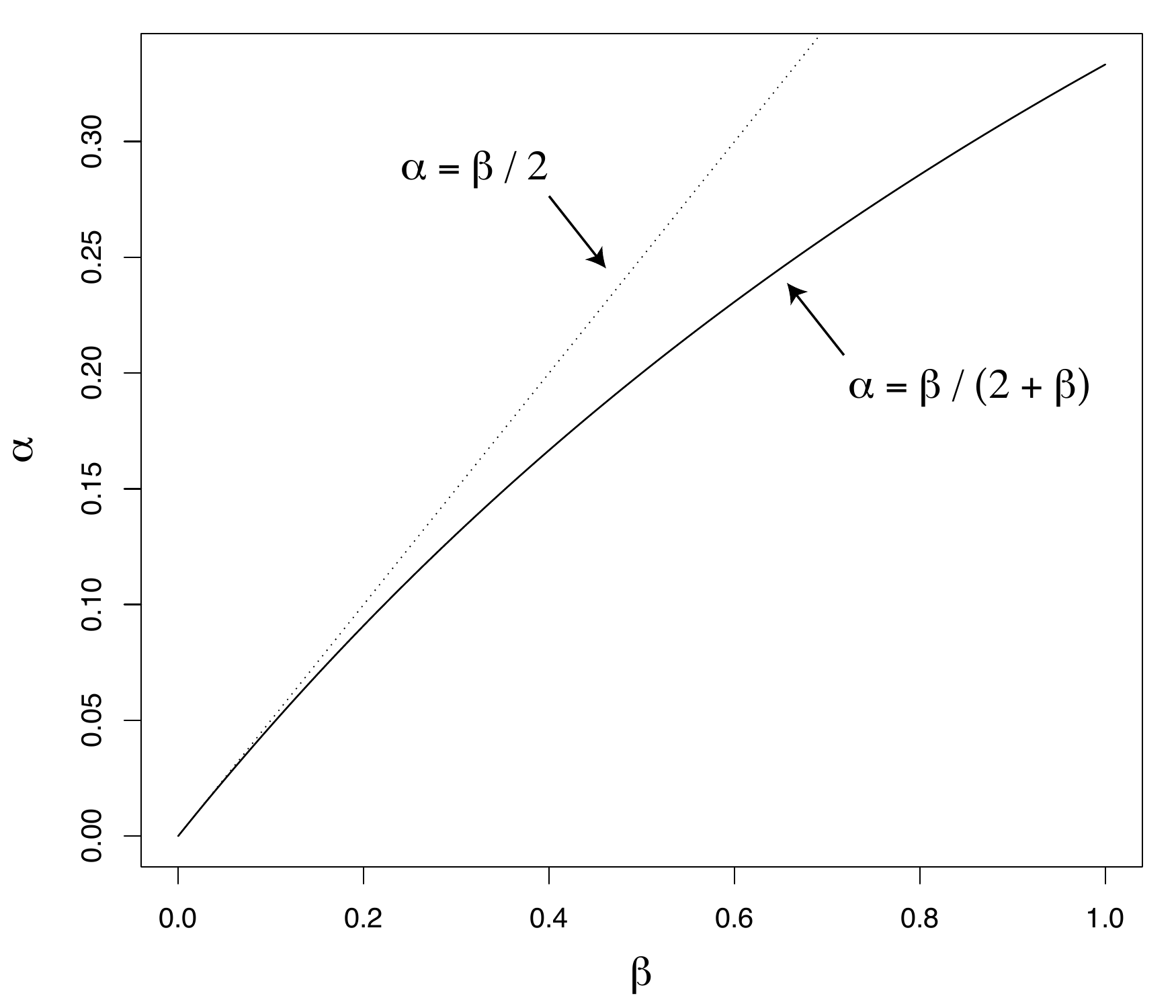}
 \caption{Relation between the tuning parameter $\alpha$ and the bound of FDR $\beta$. }
 \label{fig:fdr:alpha}
\end{figure}

\begin{figure}[!h]
\centering
\includegraphics[width=0.95\columnwidth]{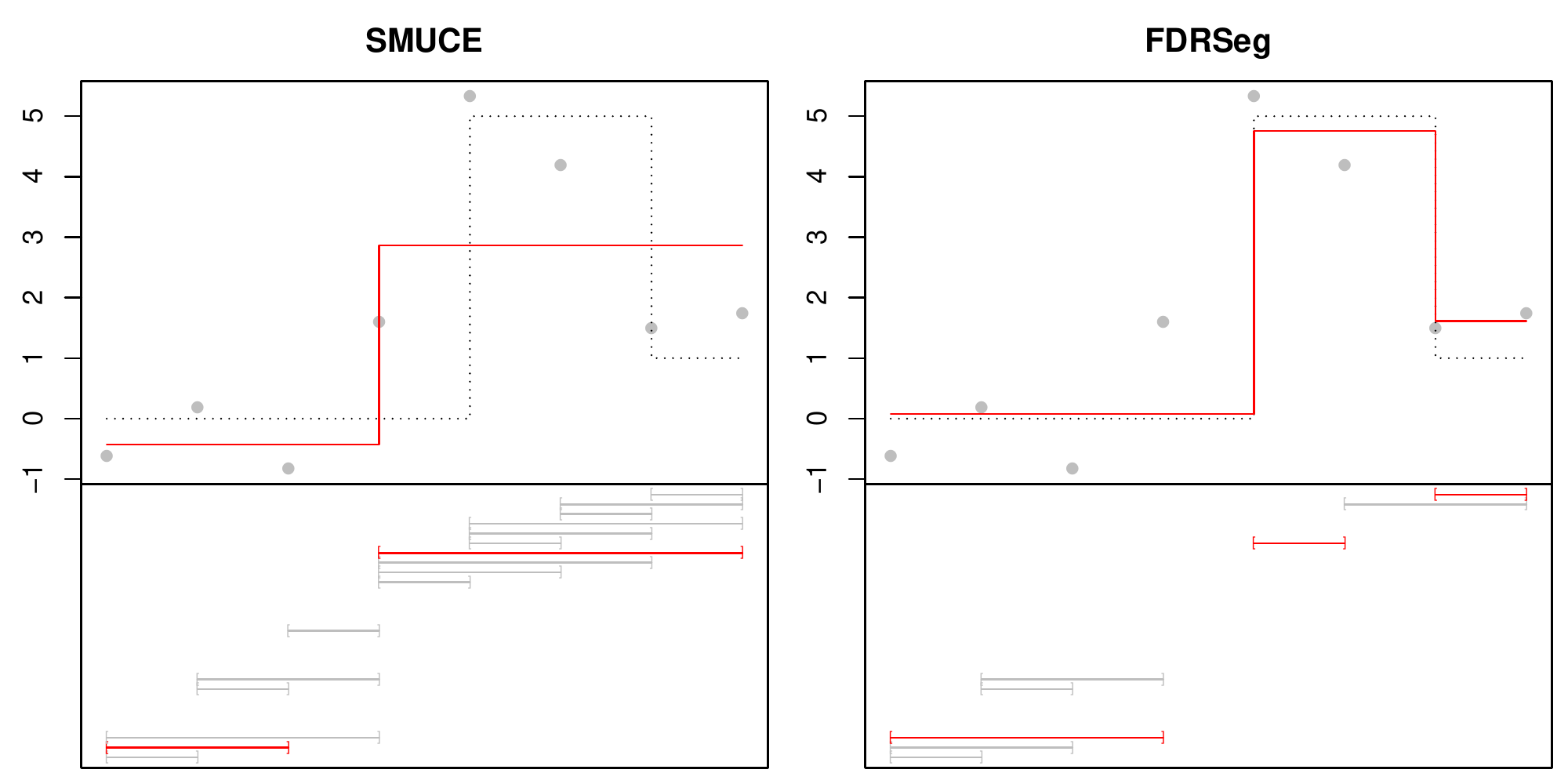}
\caption{Difference between SMUCE and FDRSeg. The upper plots show the two estimates (solid line), respectively, together with the truth (dotted line) and the data (points). The lower left (right) shows all the intervals on which there is a constant function satisfying the multiscale side-constraint of SMUCE (FDRSeg),  with red ones chosen by the estimator, separately.}
\label{pic:diff_fs_s}
\end{figure}

\begin{rem}[Comparison of SMUCE and FDRSeg]
Let us stress some notable differences to SMUCE~\citep{FriMunSie14}, which is based on restricting possible estimators to 
\begin{equation*}
\mathcal{C}_k^0 =  \left\{\mu =  \sum_{i=0}^{k} c_i \mathbf{1}_{I_i}: \max_{i=0,\ldots,k} T^0_{I_i}(Y, c_i) \le \tilde{q}_{\alpha_S}\right\},
\end{equation*}
where $T^0_I(Y, c)$ is as in~\eqref{eq:test:stat:segment}, with penalty $\text{pen} ({(j-i+1)}/{n})$ instead, and 
\begin{equation}\label{eq:smue:quantile}
\tilde{q}_{\alpha_S} = \tilde{q}_{\alpha_S}(n) := \min\left\{q: \P\left\{T^0_{[0,1)}(\varepsilon, 0) > q\right\} \le \alpha_S\right\}, \quad \varepsilon \sim \mathcal{N}(0, I_n).
\end{equation}
Firstly, this penalty term underlying SMUCE  on the interval $[i/n\,,j/n]$ only relates the ratio between the number of observations in $[i/n,\, j/n]$ and all the observations, while that of FDRSeg relies on the ratio between the number of observations in $[i/n,\, j/n]$ and the corresponding segment length of ${I}$. This modification has a flavor similar to the refined Bayes information criterion type of penalty in~\citep{ZhaSie07}. Secondly, the parameter $\alpha_S$ of SMUCE ensures that the true signal lies in the side-constraint $\mathcal{C}^0_K$ with probability at least $1-\alpha_S$. In contrast, FDRSeg considers constant parts of the true signal individually, guaranteeing that the mean value of each segment $I_i$ lies in its associated side-constraint in $\mathcal{C}_K$ with probability at least $1-\alpha$.  This makes it much less conservative, and its error controllable in terms of FDR (see Theorem~\ref{thm:mainthm}). This is a key idea underlying FDRSeg.  For an illustration of this effect see Figure~\ref{pic:diff_fs_s}. Thirdly, {the thresholding underlying SMUCE is based on 
a global quantile. 
In contrast, for FDRSeg, the quantiles $q_{\alpha}$ in~\eqref{def:qalpha} are locally chosen according to the scale, }revealing the resulting method less conservative. Note, that $q_{\alpha}(m)$ in~\eqref{def:qalpha} and $\tilde{q}_{\alpha_S}(n)$ in~\eqref{eq:smue:quantile} are even different  when $\alpha = \alpha_S$ and $m = n$. Simulations show that $q_{\alpha}(n) < \tilde{q}_{\alpha}(n)$ for every $\alpha$ and $n$, see Figure~\ref{fig:comp:quantiles}. This again highlights that FDRSeg detects more change-points than SMUCE. 

\begin{figure}[!h]
\centering
\includegraphics[width=0.95\columnwidth]{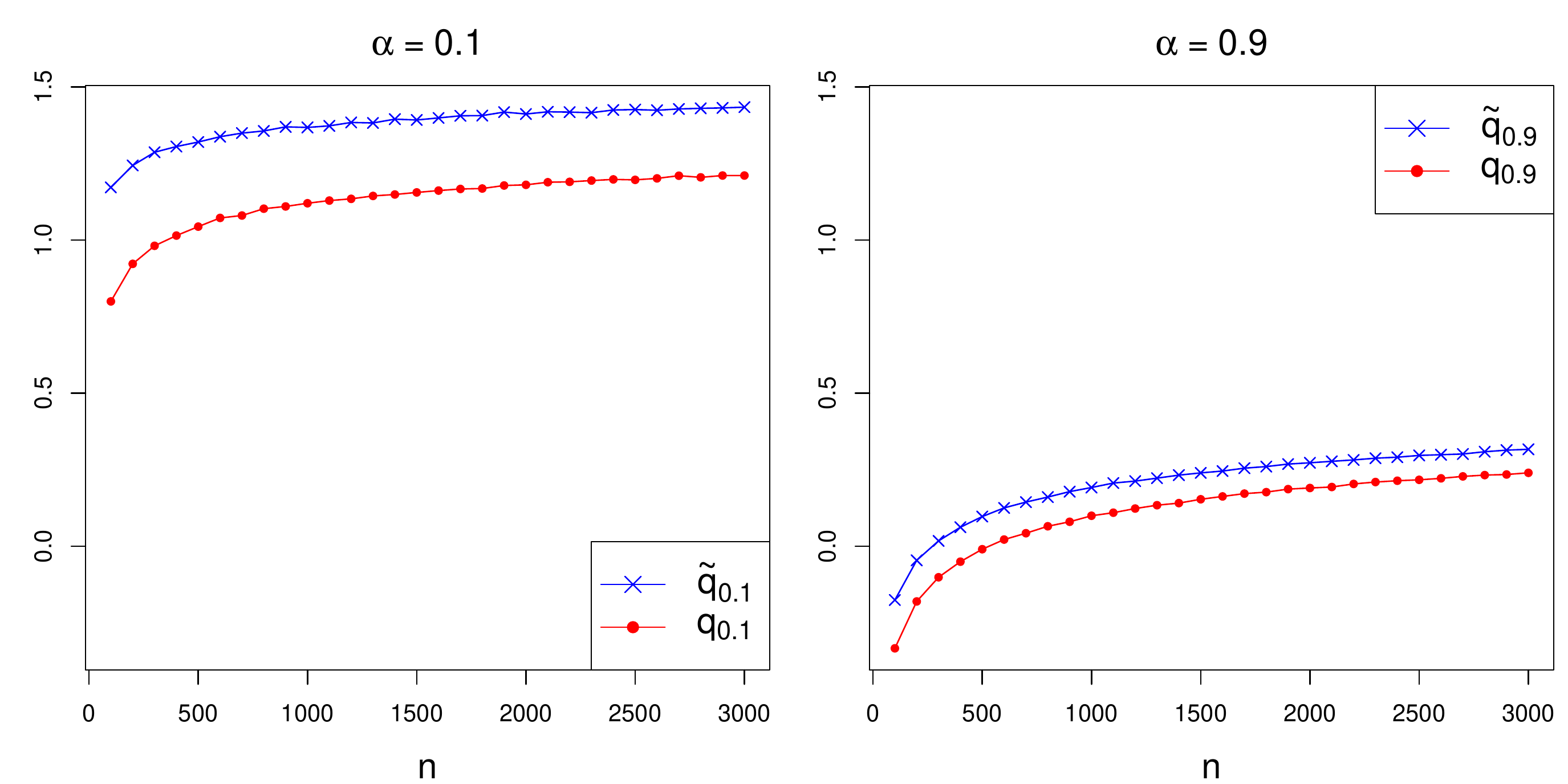}
 \caption{Comparison of $q_{\alpha}(n)$ and $\tilde{q}_{\alpha}(n)$ for various $n$. Each value is estimated by 100,000 simulations. }
 \label{fig:comp:quantiles}
\end{figure}

In situations with many change-points or low SNR, to overcome the conservative nature of SMUCE, the significance level ${\alpha_S}$ in~\eqref{eq:overestimation:bound:alpha} {to control} the overestimation error, has been suggested to be chosen close to one to produce an estimate with good screening properties~\citep{FriMunSie14}, although then the confidence statements in~\eqref{eq:overestimation:bound} and~\eqref{eq:overestimation:bound:alpha} becomes statistically meaningless.
 It follows from the arguments above that the parameter $\alpha$ of FDRSeg relates to $\alpha_S$ roughly by
\begin{equation}\label{eq:relation_choice_alpha}
1 - (1-\alpha)^{K+1} \approx \alpha_S
\end{equation}
because the probability of coverage of the true signal by $\mathcal{C}_K$ is $(1-\alpha)^{K+1}$,
where $K$ is the true number of change-points. This is confirmed by simulations. For example, consider the recovery of a teeth signal~(adopted from \citep{Fry14}) with $K = 50$ from $900$ observations contaminated by standard Gaussian noise, see Figure~\ref{fig:relation_choice_alpha_est}.  In Figure~\ref{fig:relation_choice_alpha_hist}, the histogram of estimated number of  change-points by SMUCE ($\alpha_S = 0.1$) and FDRSeg ($\alpha=0.1$) are shown in white bars from 1,000 repetitions.  It can be seen that SMUCE ($\alpha_S = 0.1$) seriously underestimates the number  of change-points, while FDRSeg estimates the right number of change-points with high probability. If we adjust $\alpha_S$ according to~\eqref{eq:relation_choice_alpha}, i.e. $\alpha_S = 1-(1-0.1)^{51} \approx 0.995$, this leads to a significant improvement of detection power of SMUCE, as is shown by the corresponding histogram of estimated number of change-points in grey bars (left panel in Figure~\ref{fig:relation_choice_alpha_hist}), however, at the expense of any reasonable statistical error control, i.e. the control of overestimating the true $K$ for SMUCE becomes increasingly more difficult as $K$ gets larger.
On the other hand, FDRSeg adapts to $K$ automatically, and works well with a choice of small values of $\beta$ in~\eqref{eq:thm:bnd}. 
Moreover, concerning the accuracy of locations, the medians of $d(\mu,\cdot)$, see~\eqref{eq:def:accLoc}, of SMUCE ($\alpha_S = 0.995$) and FDRSeg ($\alpha=0.1$) have been found as $0.0178$ and $0.0078$, respectively, while such medians conditioned on $\hat K = K$ have the same value, $0.0067$. 
This confirms the visual impression when comparing the two lower panels in Figure~\ref{fig:relation_choice_alpha_est}: Local thresholding in~\eqref{def:qalpha} and~\eqref{side-constraint:fdrsmuce} makes an important difference to SMUCE. 
\end{rem}

\begin{figure}[!h]
\centering
 \includegraphics[width=0.95\columnwidth]{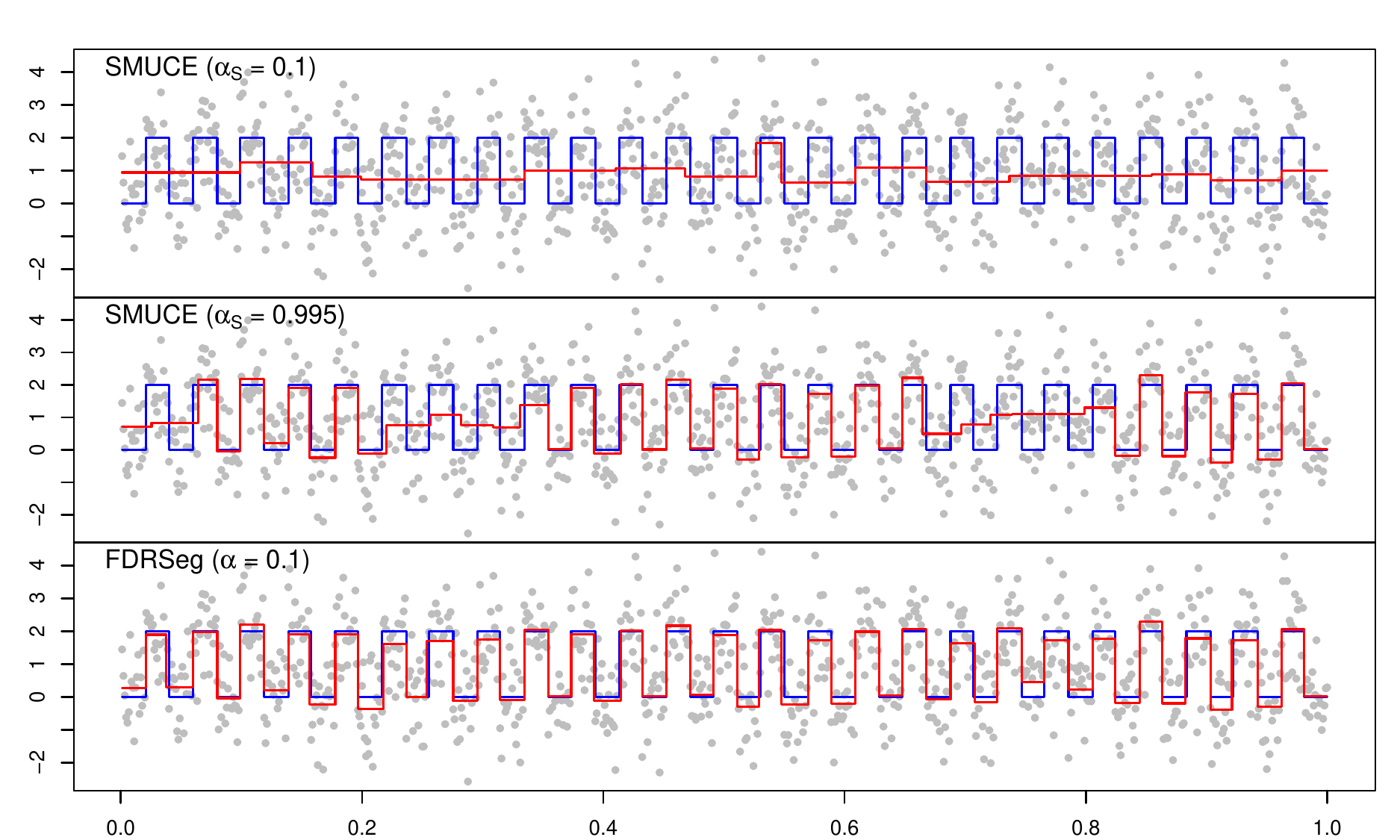}
 \caption{Estimation of teeth signal ($K = 50, \, n = 900$) by SMUCE ($\alpha_S=0.1$),  SMUCE ($\alpha_S=0.995$) and FDRSeg ($\alpha=0.1$). The true signal (blue line), together with data (points),  is shown in  each panel. }
 \label{fig:relation_choice_alpha_est}
\end{figure}

\begin{figure}[!h]
\centering
 \includegraphics[width=0.95\columnwidth]{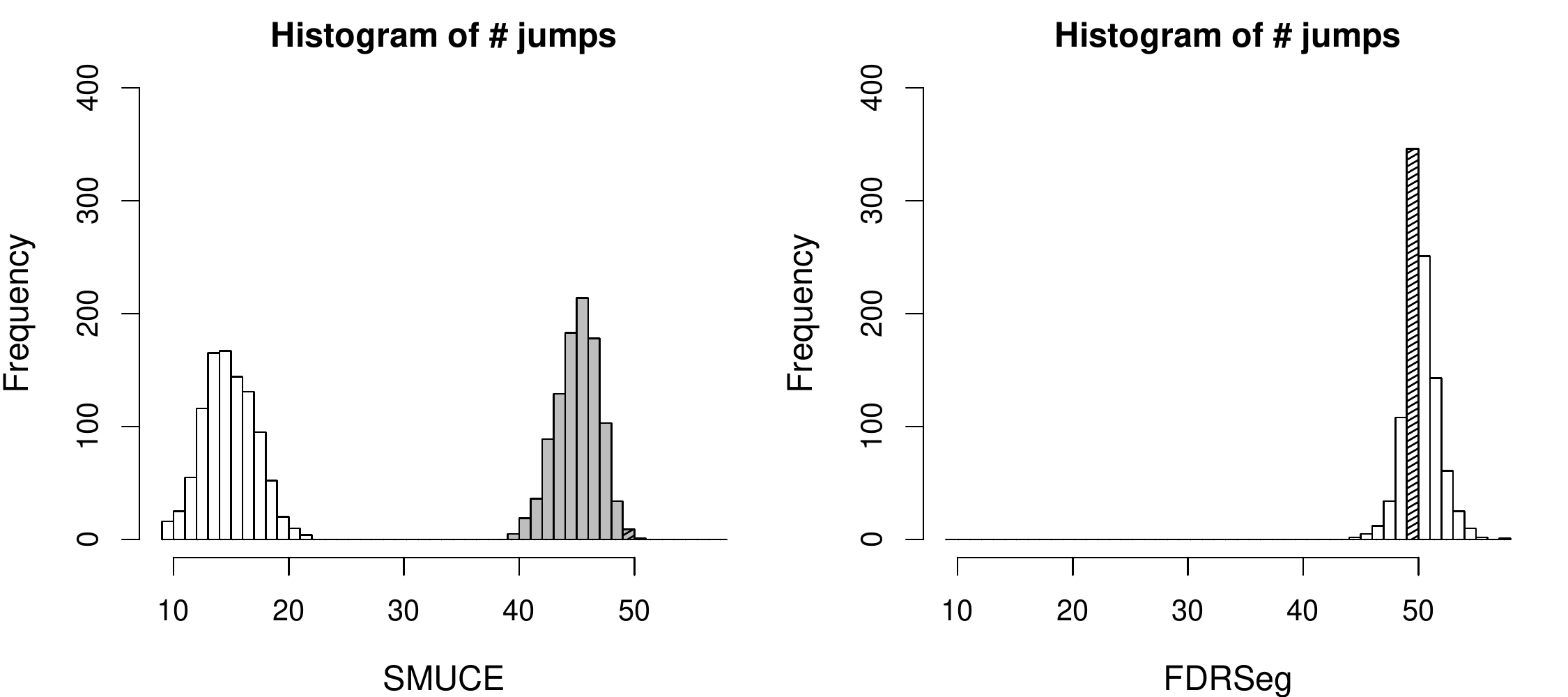}
 \caption{Histogram of number of change-points for SMUCE ($\alpha_S=0.1$, left in white bars),  SMUCE ($\alpha_S=0.995$, left in grey bars) and FDRSeg ($\alpha=0.1$, right in white bars). The shaded bars correspond to the true number of change-points $50$. The number of simulations is 1,000. }
 \label{fig:relation_choice_alpha_hist}
\end{figure}

\section{Risk bounds for FDRSeg}\label{sec:risk:fdrseg}

In order to state uniform results on the $L^p$-risk of $\hat \mu$ and on the simultaneous estimation of the change-point locations, we define the smallest segment length $\lambda_{\mu}$ of a step function $\mu$ in~\eqref{eq:true_signal} by
\[
\lambda_{\mu}  := \min_{0 \le k \le K} \abs{\tau_{k+1} - \tau_{k}},
\]
and the smallest jump size $\Delta_{\mu}$ of $\mu$ by
\[
\Delta_{\mu} :=  \min_{1 \le k \le K} \abs{c_k - c_{k-1}}, \quad \text{if } K_{\mu} \ge 1.
\]
The subscript $\mu$ will be suppressed in the following, if there is no ambiguity. Note, that no method  can recover arbitrary fine details measured in terms of {$\lambda$ and $\Delta$} for given sample size $n$. More precisely, the detection boundary for testing $\mu_n := \Delta_n \mathbf{1}_{I_n}$ against a zero signal asymptotically is given as
\begin{equation}\label{eq:test:boundary}
\frac{\Delta_n}{\sigma} \sqrt{\abs{I_n}} \ge \sqrt{\frac{2\log \abs{I_n}^{-1}}{n}} + a_n^{-1} \quad \text{ with } a_n = o(\sqrt{n}),
\end{equation}
see ~\citep{ChaWal13,FriMunSie14}. It is worth noting that FDRSeg detects such signals~\eqref{eq:test:boundary} with asymptotic power 1, {provided that the level $\alpha = \alpha_n$ is bounded away from 0,} 
as $n \to \infty$. The proof is omitted because it is similar to \citep[Theorem 5]{FriMunSie14}. 

{In the following we will show how} $\lambda$ and $\Delta$ determine the detection and estimation difficulty for step functions with multiple change-points (cf.  Theorems~\ref{thm:accuracy:location} and \ref{thm:convergence:rate:LpRisk}) {in a non-asymptotic way.} 
The following exponential bound for the estimated locations provides a theoretical justification of the previous empirical findings (see also Section~\ref{sec:sim_app}) of the good detection and estimation performance of FDRSeg. 
\begin{thm}\label{thm:accuracy:location}
{Assume the change-point regression model~\eqref{model_1} with {signal} $\mu$ in~\eqref{eq:true_signal}, and let} $(x)_+ := \max\{x, 0\}$, $\delta_{\lambda} :=  \min\{\delta,\lambda/2\}$, and $d(\mu,\hat\mu)$ defined in~\eqref{eq:def:accLoc}
{Then for  the FDRSeg $\hat\mu$ in~\eqref{def:fdr_smuce}, the following {statements} are valid: }
\begin{itemize}
\item[(i)]
It holds for any $\delta > 0$ that
\begin{align}
\P\left\{
d(\mu,\hat\mu) > \delta\right\} \le &2K \exp\left(-\frac{1}{8}\left(\frac{\Delta\sqrt{n\delta_{\lambda}}}{2\sigma}-\max_{m \le n} q_{\alpha}(m) -\sqrt{2\log \frac{e}{\delta_{\lambda}}}\right)_+^2\right) \nonumber\\
&+2K\exp\left(-\frac{n\Delta^2\delta_{\lambda}}{8\sigma^2}\right). \label{eq:bound:location:accuracy}
\end{align}
\item[(ii)]
{Let $K \ge 1$, $\alpha \equiv \alpha_n \gtrsim n^{-\gamma}$ with $\gamma\ge0$}, and assume  
\begin{equation*}
\delta_{\lambda} \ge \frac{8\left((\sqrt{\gamma}+1)\sqrt{\log n} + 2\sqrt{\log K}\right)^2}{n \min\left\{(\Delta/\sigma)^2, 1\right\}}(1+\epsilon),
\end{equation*}
for some positive $\epsilon$ independent of $n$, then 
\[
\lim_{n \to \infty} \P\left\{
d(\mu,\hat\mu_n) > \delta\right\} = 0.
\]
In particular, for every $C > 8(\sqrt{\gamma}+3)^2$, we have
\[
\lim_{n \to \infty} \sup_{\mu \in A_{n}} \P\left\{d(\mu,\hat\mu_n) > C\frac{\log n}{n}\right\} = 0,
\]
where
\[
A_{n}:=\left\{\text{step signal }\mu {\text{ in \eqref{eq:true_signal}, s.t.}} \, \lambda_{\mu}\ge 2C\frac{\log n}{n} \text{ and } \Delta_{\mu} \ge \sigma, \text{ if } K_{\mu} \ge 1\right\}.
\]
\end{itemize}
\end{thm}
{
\begin{proof}
See Appendix~\ref{sec:proof_location}.
\end{proof}
}

\begin{rem}
It is worth noting that the first term in~\eqref{eq:bound:location:accuracy} is always greater than the second one, and that the influence of $\alpha$ (or equivalently $\beta$, see~\eqref{eq:thm:bnd}) only appears in $\max_{m \le n} q_{\alpha}(m)$, {which is  bounded by $C + \sqrt{2\log(1/\alpha)}$, see Lemma~\ref{lem:bnd:quantile}.} {Hence,}
for a fixed regression function $\mu$, FDRSeg is able to estimate the jump locations correctly at a $\log (n)/n$ rate. Note, that this is the optimal sampling rate $1/n$ (up to a log-factor). It improves several results obtained for other methods, e.g. in~\citep{HarLev10} for a total variation penalized estimator a $\log^2 (n)/n$ rate has been shown. Theorem~\ref{thm:accuracy:location} also applies for a sequence of signals $\mu_n$ with $K = K_n$, $\Delta = \Delta_n$, and $\lambda = \lambda_n$. {For example, it shows that FDRSeg detects jump locations at a $\log (n)/n $ rate for $\mu_n$ with possibly unbounded $K_n$, $\lambda_n \sim \log (n) / n$, and bounded $1/\Delta_n $.} The same rate is shown for WBS in~\citep{Fry14}, however, under the additional assumption of bounded $K_n$ or {an oracle} choice of the threshold depending on the underlying signal. 
\end{rem}

Next we will study the convergence rate of FDRSeg  in terms of $L^p$-risk. By {$\tilde{\Delta}_\mu$} we denote the largest jump size of a step function $\mu$, that is,
\[
\tilde{\Delta} = \tilde{\Delta}_{\mu} :=\max_{1 \le k \le K_{\mu}} \abs{c_k - c_{k-1}} \quad\text{if } K_{\mu} \ge 1. 
\] 
Let us {introduce} the following class of step functions, $B_{\nu, \epsilon, L}$, with bounded minimal segment length and jump size:
\begin{equation}\label{eq:def:difficulty:class}
B_{\nu, \epsilon, L} := \{\text{step signal }\mu: \lambda_{\mu} \ge \nu, \text{ and } \epsilon \le \Delta_{\mu} \le \tilde{\Delta}_{\mu} \le L \text{ if } K_{\mu} \ge 1\},
\end{equation}
for $0 < \nu < 1/2$, and $0 < \epsilon < L < \infty$.  Within such classes, we {obtain} a uniform control on the $L^p$-risk of FDRSeg for $1 \le p < \infty$.
\begin{thm}\label{thm:convergence:rate:LpRisk}
Assume $B_{\nu,\epsilon,L}$ is defined in~\eqref{eq:def:difficulty:class}, and  $\hat\mu_{n, \alpha_n}$ the FDRSeg estimator with $\alpha = \alpha_n$ from $n$ observations in model~\eqref{model_1}. 
\begin{itemize}
\item[(i).] 
If {$\alpha_n \gtrsim 1/n$} and 
\[
\alpha_n = o\left(\frac{1}{\sqrt{n}}\left(\frac{\log n}{n}\right)^{\min\{1/2, 1/p\}}\right),
\]
then
\[
\mathop{\lim\sup}_{n \to \infty}\sup_{\mu \in B_{\nu,\epsilon, L}}\mathbf{E}\left[{\norm{\hat\mu_{n, \alpha_n}-\mu}_{L^p} }\right] \left(\frac{\nu\epsilon^2 n}{\sigma^2\log n}\right)^{\min\{1/2, 1/p\}} \le 25 L,
\]
for any $\sigma > 0$, $0 < \nu < 1/2$, $0 < \epsilon < L < \infty$, and $1 \le p < \infty$. 
\item[(ii).] 
If $a n^{-\gamma} \le \nu:=\nu_n \le b n^{-\gamma}$ with constants $a,b>0$, $0 < \gamma < 1$, and {$\alpha_n \gtrsim n^{-3/2}$} and
\[
\alpha_n = o\left(\frac{1}{n^{\gamma+1/2}}\left(\frac{\log n}{n^{1-\gamma}}\right)^{\min\{1/2, 1/p\}}\right),
\]
then
\[
\mathop{\lim\sup}_{n \to \infty}\sup_{\mu \in B_{\nu_n,\epsilon, L}}\mathbf{E}\left[{\norm{\hat\mu_{n, \alpha_n}-\mu}_{L^p} }\right] \left(\frac{\epsilon^2 n^{1-\gamma}}{\sigma^2\log n}\right)^{\min\{1/2, 1/p\}} \le 34 L,
\]
for any $a, b, \sigma > 0$, $0 < \gamma < 1$, $0 < \epsilon < L < \infty$, and $1 \le p < \infty$. 
\end{itemize}
\end{thm}
\begin{proof}
See Appendix~\ref{sec:proof_convergence_rate}. 
\end{proof}
In fact, the rates above are minimax optimal, possibly up to a log-term.
\begin{thm}\label{thm:convergence:rate:lower:bound}
Assume the change-point regression model~\eqref{model_1}, and $B_{\nu,\epsilon,L}$ is defined in~\eqref{eq:def:difficulty:class}.  
\begin{itemize}
\item[(i).] There is a positive constant $C$, such that 
\[
\inf_{\hat{\mu}_n}\sup_{\mu \in B_{\nu,\epsilon, L}}\mathbf{E}\left[{\norm{\hat\mu_n-\mu}_{L^p} }\right] \ge C\left(\frac{\sigma^2}{n}\right)^{\min\{1/2,1/p\}},
\]
for any $\sigma >0$, $0 < \nu < 1/2$, $0 < \epsilon < 1 < L < \infty$, and $1 \le p < \infty$. 
\item[(ii).]
If $a n^{-\gamma} \le \nu:=\nu_n \le b n^{-\gamma}$ with constants $a,b>0$, $0 < \gamma < 1$, then there is a positive constant $C$, such that
\[
\inf_{\hat{\mu}_n}\sup_{\mu \in B_{\nu_n,\epsilon, L}}\mathbf{E}\left[{\norm{\hat\mu_n-\mu}_{L^p} }\right] \ge C\left(\frac{\sigma^2}{bn^{1-\gamma}}\right)^{\min\{1/2,1/p\}},
\]
for any $a, b, \sigma > 0$, $0 < \gamma < 1$, $0 < \epsilon < 1 < L < \infty$, and $1 \le p < \infty$. 
\end{itemize}
\end{thm}
\begin{proof}
See Appendix~\ref{sec:proof_lower_bound}
\end{proof}

\section{Implementation}\label{sec:impl}
It will be shown that FDRSeg  can be efficiently computed by a specific dynamic programming (DP) algorithm, which is significantly faster than the standard DP. For convenience let us introduce
\[
\mathcal{I}\left([\frac{i}{n},\frac{j}{n})\right) = 
\begin{cases}
1 & \text{ if } T_{[i/n, j/n)}(Y,c)\le q_{\alpha}(j-i)\text{ for some constant } c,\\
0 & \text{ otherwise.}
\end{cases} 
\] 
We first consider the computation of  $\hat{K}$, see~\eqref{def:Khat}. Let $\hat K[i]$ be the estimated number of change-points by FDRSeg  when applying to $(Y_0, \ldots, Y_{i-1})$, i.e., 
\begin{equation*}
\begin{split}
\hat{K}[i] := &\min\left\{k: \max_{0\le j \le k} T_{I_j}(Y,c_j) - q_{\alpha}(\#{I_j}) \le0, \vphantom{\text{ for some } \mu = \sum_{j=0}^{k} c_j\mathbf{1}_{I_j} \text{ with } \biguplus_{j=0}^{k} I_j = [0,\frac{i}{n})}\right. \\
&\left.\hphantom{\min\left\{k: \vphantom{\max_{0\le j \le k} T_{I_j}(Y,c_j) - q_{\alpha}(\#{I_j}) \le0, }\right.} \text{ for some } \mu = \sum_{j=0}^{k} c_j\mathbf{1}_{I_j} \text{ with } \biguplus_{j=0}^{k} I_j = [0,\frac{i}{n})\right\} 
\end{split}
\end{equation*}
for $i=1,\ldots, n$, where $\biguplus$ denotes disjoint union.  Then the estimated number of change-points $\hat{K}$ in~\eqref{def:Khat} is given by $\hat{K}[n]$. It can be shown that the following recursive relation
\begin{equation}\label{recursive:a}
\begin{split}
\hat{K}[0] &:= -1 \\
\hat{K}[i] &= \min\left\{\hat{K}[j]+1: \mathcal{I}\left([\frac{j}{n}, \frac{i}{n})\right) = 1,\, j = 0,\ldots,i-1 \right\}
\end{split}
\end{equation}
holds for $i = 1,\ldots, n$.  Eq.~\eqref{recursive:a} is often referred to as \emph{Bellman equation}~\citep{Bel57}, also known as \emph{optimal substructure} property in computer science community~(\citealp{CorLeiRivSte09}).  It justifies the use of \emph{dynamic programming}~\citep{Bel57, BelDre62} for computing FDRSeg. In this way, the computation of $\hat{K}$ is decomposed into smaller subproblems of determining $\hat{K}[i]$'s. For each subproblem, it boils down to checking the existence of constant functions which satisfy the multiscale side-constraint on $[j/n,i/n)$ i.e. $\mathcal{I}\left([j/n,i/n)\right) {=} 1$. The $\hat{K}[i]$ is computed, via the recursive relation~\eqref{recursive:a}, as $i$ increases from $1$ to $n$. For each $i$, this involves the search space of $\{0,\ldots, i-1\}$, which increases as $i$ approaches $n$.
However, some of such searches are, actually, not necessary and can be pruned. This can be seen by rewriting the recursive relation in terms of the number of change-points.  
Let $\mathcal{A}_{0} := \{0\}$ and $\mathcal{B}_{0} := \{1,2,\ldots,n\}$. For $k = 1, 2, \ldots,$ let 
\begin{align*}
r_k& := \max\left\{j:  T^0_{[i/n,j/n]}(Y, c)  \le \max_m q_{\alpha}(m)\text{ for some } i \in \mathcal{A}_{k-1},  c \in \R\right\}, \\
\mathcal{A}_k&:=\left\{i \in\mathcal{B}_{k-1}\cap [1, r_k]: \mathcal{I}\left([j/n, i/n)\right) = 1  \text{ for some } j \in \mathcal{A}_{k-1}\right\}, \\
\mathcal{B}_k&:=\mathcal{B}_{k-1}\setminus\mathcal{A}_{k}.
\end{align*}
Then $\hat{K} = k^*-1$ with $\mathcal{A}_{k^*} \ni n$. The reason for introducing $r_k$ is that there is no need to consider larger intervals if the multiscale side-constraint on an interval does not allow a constant signal even with the maximal penalty and the maximal quantile. Now for each $i$ we only need to search in a subset $\mathcal{B}_{k}\cap[1,r_k]$ of $\{0, \ldots, i-1\}$, where $k:=k(i)$.  The complexity for computing $\hat{K}$ is bounded from above by
\begin{equation}\label{eq:computation_complexity}
\begin{split}
&\mathcal{O}\left( \sum_{k=0}^{\hat{K}} (\#{\mathcal{A}_{k}})\biggl(r_{k+1}-\min\mathcal{A}_{k}-\frac{\#{\mathcal{A}_{k}}}{2}\biggr)^2\right)  \\
& \qquad\qquad \le  \mathcal{O}\left(n\max_{0 \le k \le \hat K}{\biggl(r_{k+1}-\min\mathcal{A}_{k}-\frac{\#{\mathcal{A}_{k}}}{2}\biggr)^2}\right).
\end{split}
\end{equation}
The value $\max_{0 \le k \le \hat K}{(r_{k+1}-\min\mathcal{A}_{k}-(\#{\mathcal{A}_{k}})/2)^2}$ depends on the signal and the noise. If the signal has many change-points and segments have similar lengths, it is probably a constant independent of $n$.  The higher the noise level, the larger it might be.  In such situation, the computation complexity is linear, although in the worst case it can be cubic in $n$.  

Indeed, the searches of $\hat{K}$ and the maximum likelihood estimate can be done simultaneously, if we record the likelihood for each point $i$. The complexity is again bounded above by~\eqref{eq:computation_complexity} but with a possibly larger constant. The memory complexity of the whole algorithm is linear, i.e. $\mathcal{O}(n)$. 
%
We omit technical details, and provide the implementation in the R package ``FDRSeg'' (\url{http://www.stochastik.math.uni-goettingen.de/fdrs}).

\section{Simulations and Applications}\label{sec:sim_app}
\subsection{Simulation study}
We now investigate the performance of  FDRSeg under situations with various SNRs {and} different number of change-points, and compare it with PELT~\citep{KillFeaEck12}, BS~\citep{ScoKno74}, CBS~\citep{OlsVenLucWig04, Ven07}, WBS~\citep{Fry14}, {and SMUCE~\citep{FriMunSie14}}. As mentioned in Section~\ref{sec:intro}, these methods represent a selection of powerful state of the art procedures from two different view points:  {first,} exact and fast {global} optimization {methods} based on dynamic programming, including PELT, and SMUCE;  {second,} {fast} greedy methods based on {local} single change-point detection, including BS, CBS and WBS.  In addition, we also include {two recent {fully automatic} penalization methods, {specifically} tailored to jump detection.} {The first is based on a} modified \emph{Schwarz information criterion (SIC)}~\citep{ZhaSie07}, referred to as mSIC, which assumes the number of change-points is bounded. {The second is} a recent variant~\citep{ZhSie12}, referred to as mSIC2, which is {primarily} designed for many change-points.
Concerning implementation, we use {the CRAN} R-packages ``PSCBS'' for CBS, ``wbs'' for BS and WBS,  ``changepoint'' for PELT, and an efficient implementation in our R-package ``FDRSeg'' for SMUCE, {see \url{http://www.stochastik.math.uni-goettingen.de/fdrs}. }
For both SMUCE and FDRSeg, we estimate the $\alpha$-quantile thresholds by 5,000 Monte-Carlo simulations. The penalty $2\log(K)$ is chosen for PELT, which is dubbed by ``SIC1'' in the codes provided by its authors, and works much better than the default choice. If we identify a change-point with two parameters (location and jump-size), this is the same as the SIC. 
We use the automatic rule, \emph{strengthened SIC}, recommended by the author for WBS. The default parameter setting provided in the packages was used for BS and CBS. 
{For mSIC and mSIC2, maximum likelihood estimates are first computed by dynamic programming~(see \citep{FriKemLieWin08}), for each fixed number of change-points up to some prechosen constant $K_{\max}$, and then the optimal solutions are found within such maximum likelihood estimates, according to criteria in \citep{ZhaSie07} and \citep{ZhSie12}, respectively. Thus, their computation complexity  depends increasingly on $K_{\max}$. }
In all simulated scenarios, we assume that the noise level $\sigma$ is known beforehand.  For quantitative evaluation, we will use the \emph{mean integrated squared error} (MISE), 
the error of estimated locations $d_*(\hat\mu) := \E {d(\mu,\hat\mu)}$, see~\eqref{eq:def:accLoc}, 
the FDR defined in~\eqref{def:fdr} and the \emph{V-measure}~\citep{Rosenberg2007}, {a segmentation evaluation measure, which} takes values in $[0,1]$, with a larger value indicating higher accuracy.  It is based upon two criteria for clustering usefulness, homogeneity and completeness, which capture a clustering solution's success in including all and only data points from a given class in a given cluster. In particular, a V-measure of $1$ shows a perfect segmentation.  All the experiments are repeated 1,000 times.

\subsubsection{Varying noise level}
Let us consider the impact of different noise levels. To this end, we use the mix signal~(adopted from \citep{Fry14}), see Figure~\ref{fig:mix_various_noise:hist:jmploc}, with additive Gaussian noise, which is a mix of prominent change-points between short intervals and less prominent change-points between longer intervals. The noise level  $\sigma$  varies from $1$ to $8$, and the number of observations $n = 560$. 
For SMUCE and FDRSeg, we choose the same parameter $\alpha_S = \alpha = 0.15$. As in Figure~\ref{fig:mix_various_noise}, FDRSeg outperforms others in all noise levels, in terms of V-measure, MISE, $d_*(\cdot)$, and detection power measured by the average number of detected change-points. As indicated by the number of detected change-points and MISE, PELT ranks second followed by WBS,  then mSIC, CBS,  SMUCE and lastly BS. The same order of performance is also seen from V-measure and $d_*(\cdot)$ up to $\sigma = 5$, but SMUCE deteriorates slower as noise level $\sigma$ increases and achieves a better V-measure and $d_*(\cdot)$ than CBS when $\sigma \ge 6$ and than WBS at $\sigma=8$. {The mSIC2 performs comparably to mSIC when $\sigma \le 3$, while deteriorating faster as $\sigma$ increases,  similar to BS when $\sigma \ge 7$.} It is worth noting that the empirical FDR of FDRSeg is around $0.1$, {below $\alpha = 0.15$ and} the theoretical bound $\approx 0.35$ {in Theorem~\ref{thm:mainthm}} (indicated by the dashed horizontal line in the lower-left panel). The CBS has the second largest empirical FDR, while that of PELT, SMUCE, mSIC, {mSIC2}, BS and WBS is almost zero. Once the quantiles for SMUCE and FDRSeg are simulated, they can be stored and used for later computations, which are therefore excluded from the recorded computation time. The computation time of FDRSeg is similar to the fastest ones, namely PELT, BS and SMUCE, at $\sigma = 1$ and increases with the noise level $\sigma$. The FDRSeg is faster than WBS and CBS in all scenarios. {As mentioned earlier,} the computation time of mSIC {and mSIC2} depends on the upper bound {$K_{\max}$} of the possible number of change-points, which is set to 100. To have a closer examination, we also illustrate histograms of the locations of change-points, for $\sigma=8$ in Figure~\ref{fig:mix_various_noise:hist:jmploc}. In this situation, the FDRSeg has {always} the largest detection power over all change-point locations.  

The constant signal with no change-point {serves as} an  example to examine whether FDRSeg detects artificial jumps. Figure~\ref{fig:const:fdrSmuce:smuce} shows the comparison between SMUCE ($\alpha_S = 0.15$) and FDRSeg ($\alpha = 0.15$) when $\mu \equiv 0$. Remarkably, the difference between the two estimators is negligible and the overestimation by FDRSeg number of jumps is quite insignificant.

\begin{figure}[!h]
\centering
\includegraphics[width= 1\columnwidth]{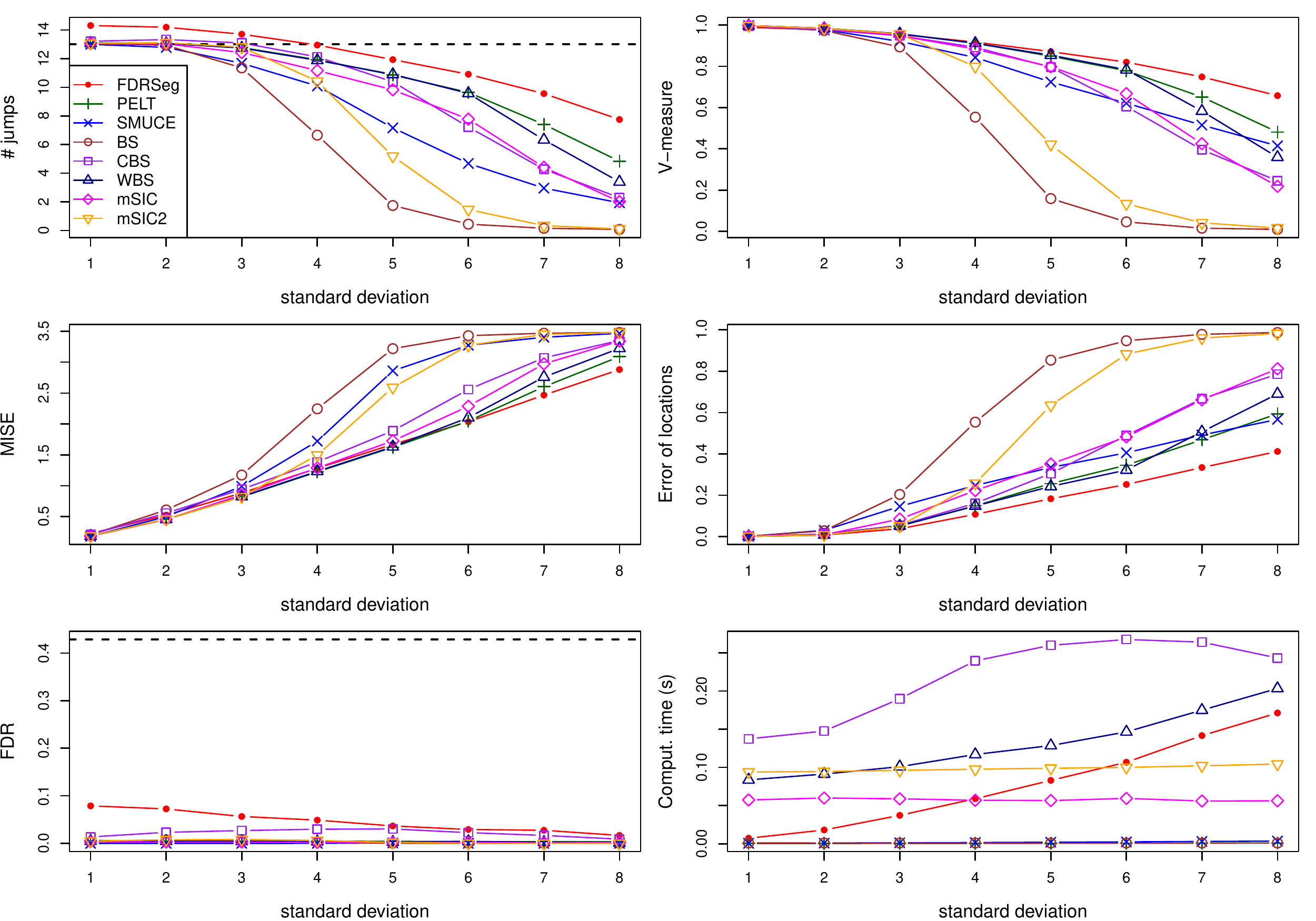} 
\caption{{The mix signal with various noise levels. True number of change-points is $K = 13$, indicated by the dashed line in the first panel. }}
 \label{fig:mix_various_noise}
\end{figure}

\begin{figure}[!h]
\centering
\includegraphics[width= 1\columnwidth]{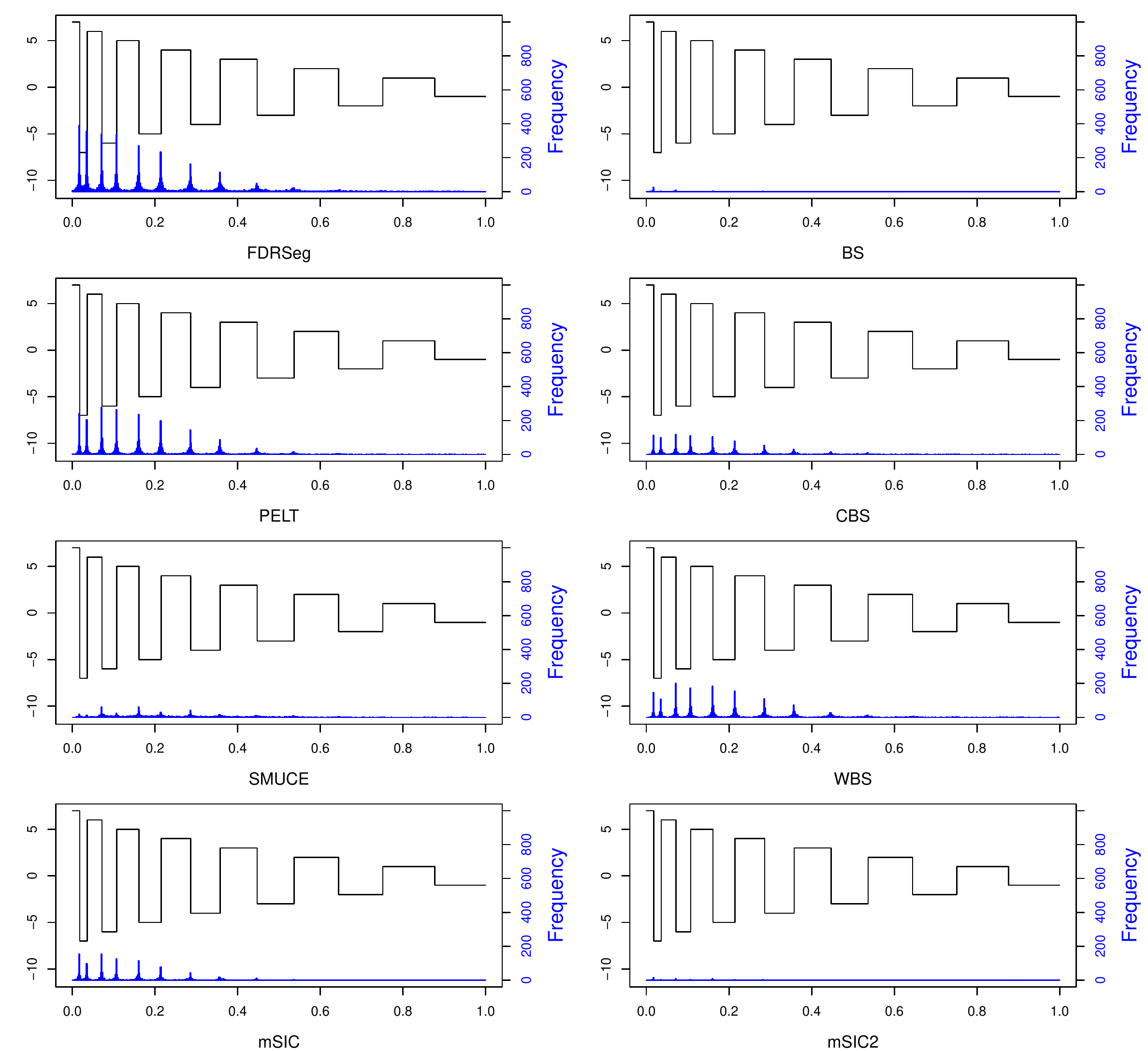} 
\caption{{The histogram of the estimated locations of  change-points for the mix signal with $\sigma=8$. As a benchmark, the true signal is plotted. }}
 \label{fig:mix_various_noise:hist:jmploc}
\end{figure}

\begin{figure}[!h]
\centering
\includegraphics[width= 0.65\columnwidth]{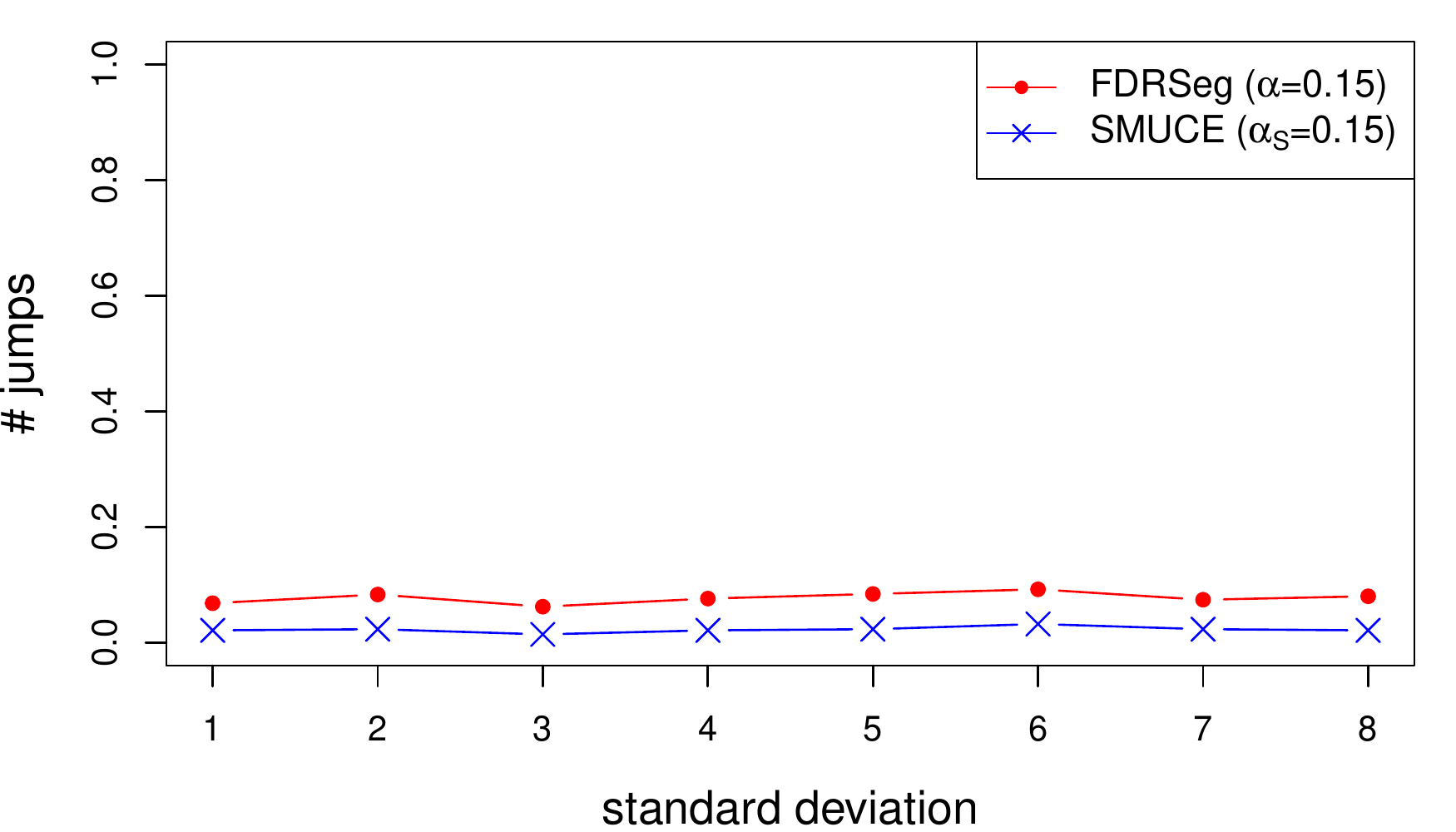} 
\caption{The constant signal with various noise levels.}
 \label{fig:const:fdrSmuce:smuce}
\end{figure}



\subsubsection{Varying frequency of change-points}

In order to evaluate the detection power as $K$ increases, we employed the teeth signal (see Figure~\ref{fig:relation_choice_alpha_est}) with $n=3$,000, and $K = n^{\theta}$, $\theta = 0.1, 0.2, \ldots, 0.9$, as its {integrated} SNR remains the same for different number of change-points. 
The same parameter $\alpha_S = \alpha = 0.1$ is chosen for SMUCE and FDRSeg. The results are summarized in Figure~\ref{fig:teeth_various_cp}. The FDRSeg, mSIC, and PELT perform comparably well in all situations in terms of number of  detected change-points, V-measure, MISE and $d_*(\cdot)$, {while FDRSeg is slightly better in terms of accuracy of change-point locations  at $\theta = 0.9$.}  As shown by V-measure, CBS and WBS fail when $\theta \ge 0.7$, BS fails when $\theta \ge 0.8$, and SMUCE {and mSIC2} deteriorate at $\theta = 0.9$. A similar trend can also be seen for the number of estimated change-points, MISE, and $d_*(\cdot)$. It is interesting that the empirical FDR of FDRSeg gets closer to the theoretical bound $\approx 0.22$ as $\theta \to 1$, {indicating that this gets sharper for increasing $K$, although we have no proof for this.} The empirical FDR of CBS is large when the change-points are sparse, and decreases as $K$ increases, while  PELT, SMUCE, {mSIC}, mSIC2, BS and WBS have a relatively small FDR close to zero in all cases. The computation time of FDRSeg decreases as $K$ increases, and is comparable to the fastest ones (SMUCE, PELT and BS), when $\theta \ge 0.6$.  The computation time of {mSIC and} mSIC2 is the slowest, {since we have to search among maximum likelihood estimates with all possible numbers of change-points, i.e. $K_{\max} = n-1$.} 

\begin{figure}[h]
\centering
\includegraphics[width= 1\columnwidth]{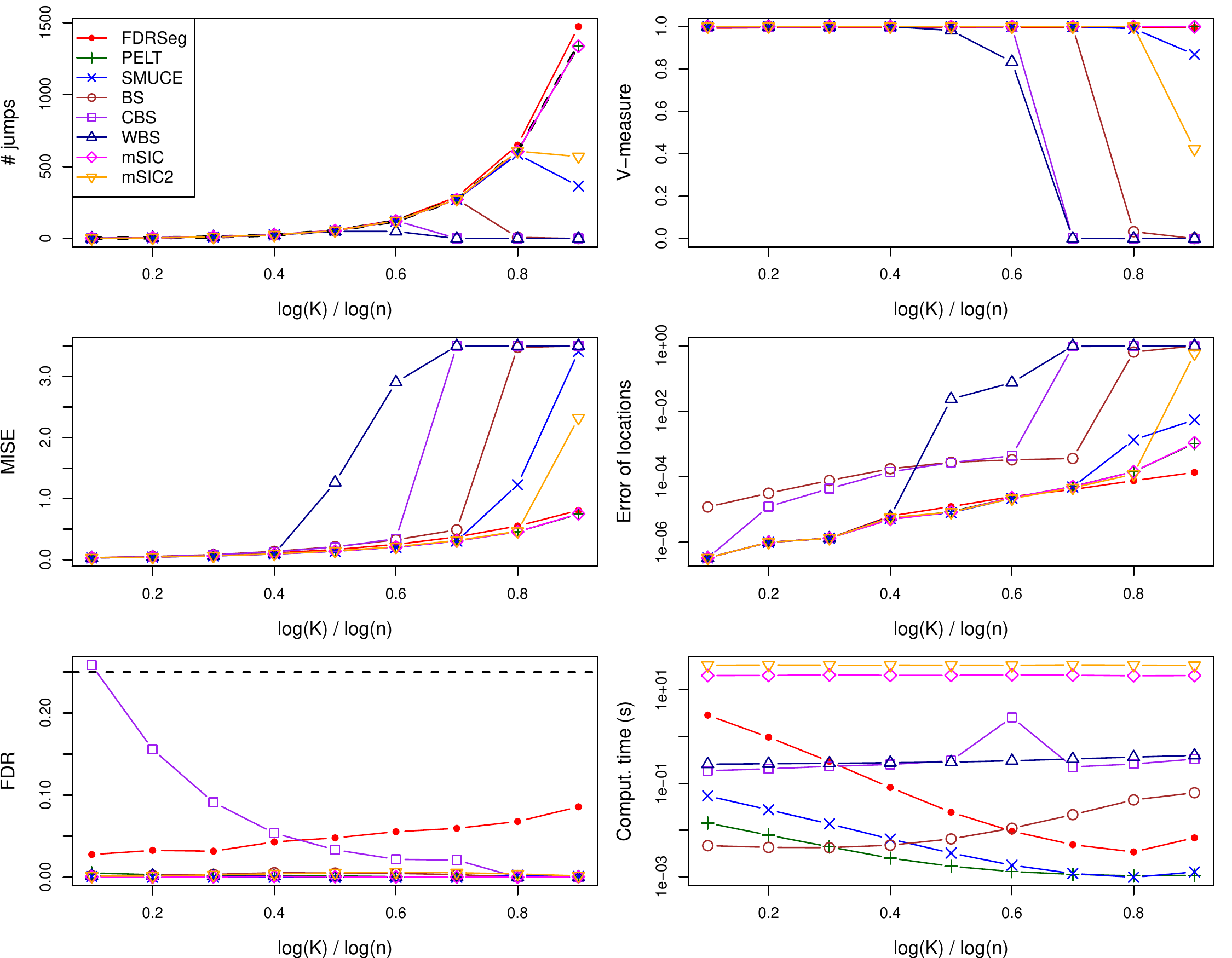} 
\caption{{The teeth signal with various frequencies of change-points. True number of change-points is plotted in dashed line in the first panel. }}
 \label{fig:teeth_various_cp}
\end{figure}

\subsection{Array CGH data}

Identifying the chromosomal aberration locations in genomic DNA samples is crucial in understanding the pathogenesis of many diseases, in particular, various cancers. Array comparative genomic hybridization (CGH) provides the means to quantitatively measure such changes in terms of DNA copy number \citep{pinkel1998high}.  The statistical task is to determine accurately the regions of changed copy number, {and the} model~\eqref{model_1} {and variants thereof} has been commonly studied in this {context} \citep{OlsVenLucWig04,ZhaSie07,TibWan08,Jen10}. We compared FDRSeg with SMUCE, and CBS, which is designed for the analysis of array CGH data, on the Coriel data set from~\citep{Sni01}.  {Following~\citep{OlsVenLucWig04} outliers have been removed} before segmentation. The noise level is estimated by an interquartile range (IQR) applied to local differences {(see \citep{DavKov01})}
\begin{equation}\label{eq:noise:level:estimation}
\hat\sigma = \frac{1.349}{\sqrt{2}}\Big(\hat{q}_{0.75}-\hat{q}_{0.25}\Big),
\end{equation}
where $\hat{q}_{\alpha}$ is the empirical $\alpha$-quantile of $\{Y_i-Y_{i-1}\}_{i=1}^{N-1}$. 
The CBS was computed using default parameters provided in the package ``PSCBS''. The estimated copy number variations by each method are plotted with the data (points) for cell line GM01524 in Figure~\ref{pic:array_cgh}.  The SMUCE ($\alpha_S = 0.05$) detects 8 change-points, while FDRSeg ($\beta = 0.05$) finds 5 more change-points, which are all found by CBS as well. The latter {provides} the largest number of change-points, 17, 4 of them are not supported by FDRSeg (marked by `x').  We stress that, there is biological evidence that {such} small jumps might be artifacts due to genomic waves, see~\citep{DLHYGHBMW08}. The model~\eqref{model_1} apparently does not take such waves into account.  Apart from {this possible} modeling error, it is worth noting that, by Theorem~\ref{thm:mainthm}, among 13 change-points by FDRSeg there are on average at most 0.7 false ones.  In order to study the robustness against {such a modeling error}, we consider step functions {in~\eqref{model_1}} with periodic trend component, as in~\citep{OlsVenLucWig04,ZhaSie07}, i.e.
\begin{equation}\label{eq:test:robust:acgh}
Y_i \sim \mathcal{N}\left( \mu({i}/{n}) + 0.25 b \sin(a \pi i), \sigma^2\right),\quad i = 0, 1, \ldots, n-1,
\end{equation}
where $\sigma = 0.2$, $n = 497$, and $\mu$ has change-points $\{137, 224, 241, 298, 307, 331\}/n$ with values $\{-0.18,0.08,1.07,-0.53,0.16,-0.69,-0.16\}$ on each segment, respectively. The FDRSeg with $\beta = 0.05$ is applied to {the} signal~\eqref{eq:test:robust:acgh} within a range of $a$ and $b$. The frequency of detecting the right number of change-points, {together with the average of $(\hat K - K)$}, in 1,000 simulations is given in Figure~\ref{pic:stable:analysis:acgh}. It shows that FDRSeg is robust within a large range of local trends, and {only} {includes false positives} when the trend {becomes} large and highly oscillating.

\begin{figure}[t]
\centering
\includegraphics[width= 1\columnwidth]{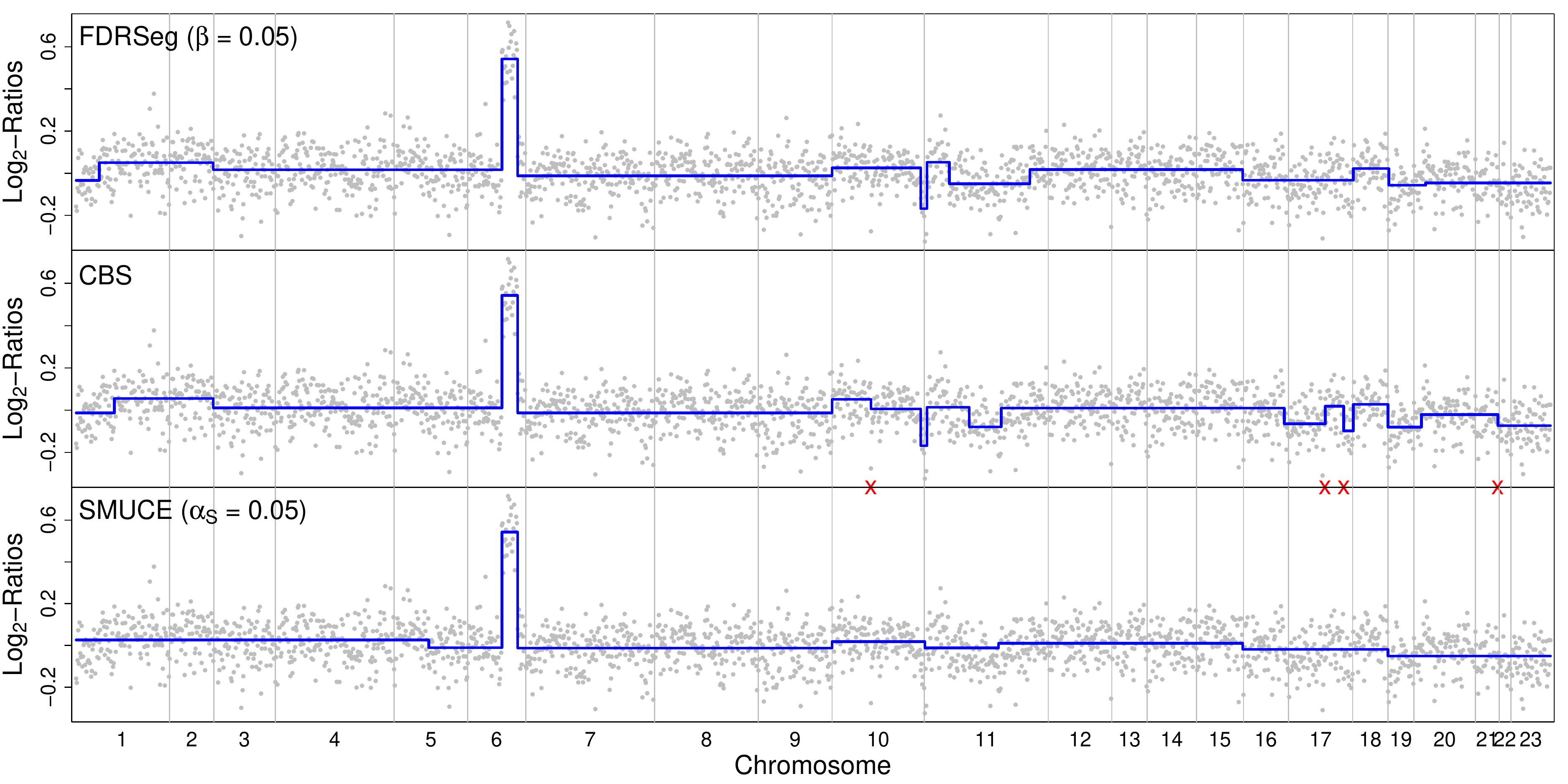} 
\caption{Array CGH profile in GM01524 cell line in the Coriel data set. 
}
 \label{pic:array_cgh}
\end{figure}

\begin{figure}[t]
\centering
\includegraphics[width= 1\columnwidth]{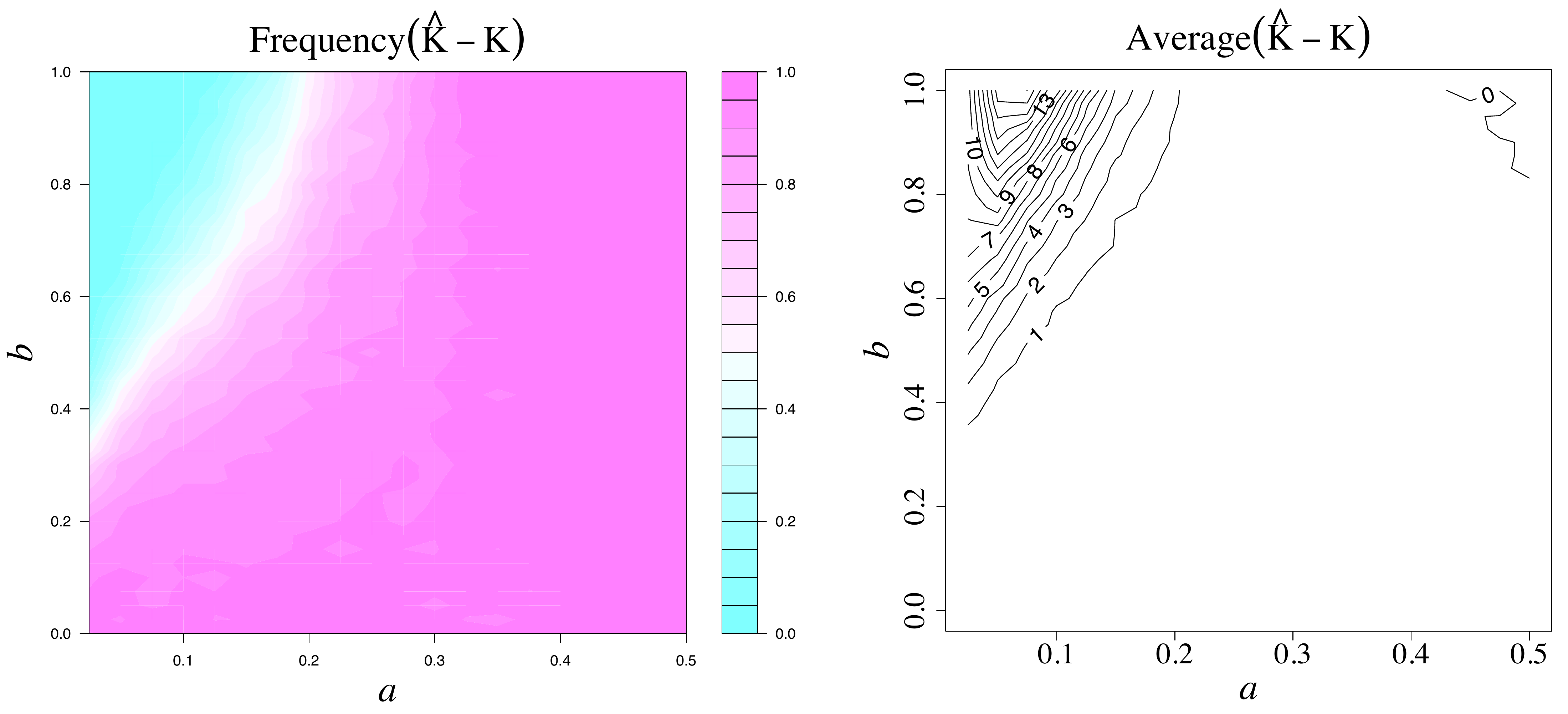} 
\caption{{Frequencies of estimating correctly the number of jumps (left), and averages of $({\hat K} - K)$ (right),  by FDRSeg ($\beta = 0.05$) for signal~\eqref{eq:test:robust:acgh} with various $a$ and $b$ as in~\eqref{eq:test:robust:acgh}. } }
 \label{pic:stable:analysis:acgh}
\end{figure}

\subsection{Ion channel idealization}\label{subsec:ion:channel}

Being prominent components of the nervous system, ion channels play major roles in cellular processes~\citep{hille2001ion}, which are helpful in diagnosing many human diseases such as epilepsy, cardiac arrhythmias, etc.~\citep{kass2005channelopathies}. The data analysis is to obtain information about channel characteristics and the effect of external stimuli by monitoring their behavior with respect to conductance and/or kinetics~\citep{chung2007biological}. The measuring process involves an analog low-pass filter prior to digitization. As suggested by~\citep{Hot12}, {hence} a realistic model for observations is
\begin{equation}\label{ion_channel:real_model}
Y_i = (\rho * \mu)({i}{\vartheta}) + \tilde{\varepsilon}_i,
\end{equation}
where $1/\vartheta$ is the sampling rate, and the convolution kernel $\rho$ of the low-pass filter has compact support in an interval of length $L$, such that $\int \rho(t) dt = 1$. Being the independent and identically distributed (i.i.d.) Gaussian noise after the low-pass filter $\rho$, the $\tilde{\varepsilon}_i$'s are still Gaussian with mean zero, but are correlated now. 

As mentioned earlier in  Section~\ref{sec:intro}, FDRSeg can be extended to more general models than~\eqref{model_1}. {We illustrate this for the present case of colored noise. To this end, we modify FDRSeg which explicitly takes into account the dependence structure of the noise in~\eqref{ion_channel:real_model}. This requires to} adjust the definition of quantiles $q_{\alpha}(\cdot)$ by using dependent Gaussian random variables, see~\eqref{def:qalpha}. Note that the dependence structure is completely known from the kernel $\rho$, so the modified quantiles can also be estimated via Monte-Carlo simulations. {In order to analyze the data properly,} we observe that $\rho * \mu$ is constant on $[s+\vartheta L, t]$ if $\mu$ is constant on $[s, t]$. Thus we consider only intervals contained in $[\hat{\tau}_i+\vartheta L, \hat{\tau}_{i+1})$ in the multiscale side-constraint~\eqref{side-constraint:fdrsmuce} instead of all subintervals of $[\hat{\tau}_i, \hat{\tau}_{i+1})$, for $i = 0,1,\ldots, \hat K$.  By incorporating these two modifications, we {obtain a modified version of FDRSeg adjusted to this dependency, \emph{D-FDRSeg}.} For comparison, we consider the \emph{jump segmentation by multiresolution filter} (J-SMURF) estimator~\citep{Hot12}. The implementation of J-SMURF is provided in R-package ``stepR'', available from CRAN. As in~\citep{Hot12}, the significance level $\alpha_J$ of J-SMURF is set to 0.05. The significance parameter $\beta$ of D-FDRSeg is also chosen as 0.05. The noise level, i.e. the standard deviation of $\tilde{\varepsilon}_i$, is estimated by~\eqref{eq:noise:level:estimation} from the undersampled data $\{Y_{iL}\}_i$.

In order to explore the potential of D-FDRSeg, we first carried a validation study on simulated data.  
{Mimicking various dynamics of ion channels, we choose the truth $\mu$ in~\eqref{ion_channel:real_model} by a simulated continuous time two-state Markov chain with different transition rates for 1 s. The true signal was tenfold oversampled at 100 kHz, and added by Gaussian white noise. Then, a digital low-pass filter with kernel $\rho$ in~\eqref{ion_channel:real_model} was applied, and the data with 10,000 points were finally obtained after a subsampling at 10 kHz. The noise level was chosen such that SNR equals to 3. All the parameters above are typical for a real experimental setup~(see~\citep{vandongen96,Hot12} for further details).} 
The average of  $(\hat K - K)$ for J-SMURF, FDRSeg, and D-FDRSeg  in 100 simulations is given in Figure~\ref{pic:ion_channel:sim:study}. {As to be expected,} FDRSeg detects {a large amount} of false positives due to violation of the independence assumption {of the noise~\eqref{ion_channel:real_model}}, while D-FDRSeg with the dependence adjustment {corrects for this. It} shows a higher accuracy, and a higher detection power than J-SMURF, over all transition rates. 
We further compare D-FDRSeg with J-SMURF on experimental data: a characteristic conductance trace of gramicidin A (provided by the Steinem lab, Institute of Organic and Biomolecular Chemistry, University of G{\"o}ttingen) with a typical SNR, $L = 30$ and $\vartheta = 0.1$ ms, see Figure~\ref{pic:ion_channel:dfdr:jsmurf}.  The J-SMURF detects only 8 change-points, while FDRSeg {suggests} 5 additional ones, i.e. 13 in total, all of which are visually reasonable.  This illustrates the ability of FDRSeg  to detect change-points simultaneously {over} various scales, as it is {required} for the investigated gramicidin channel~(see \citep{Hot12} for an explanation).

\begin{figure}[!h]
\centering
\includegraphics[width= 0.6\columnwidth]{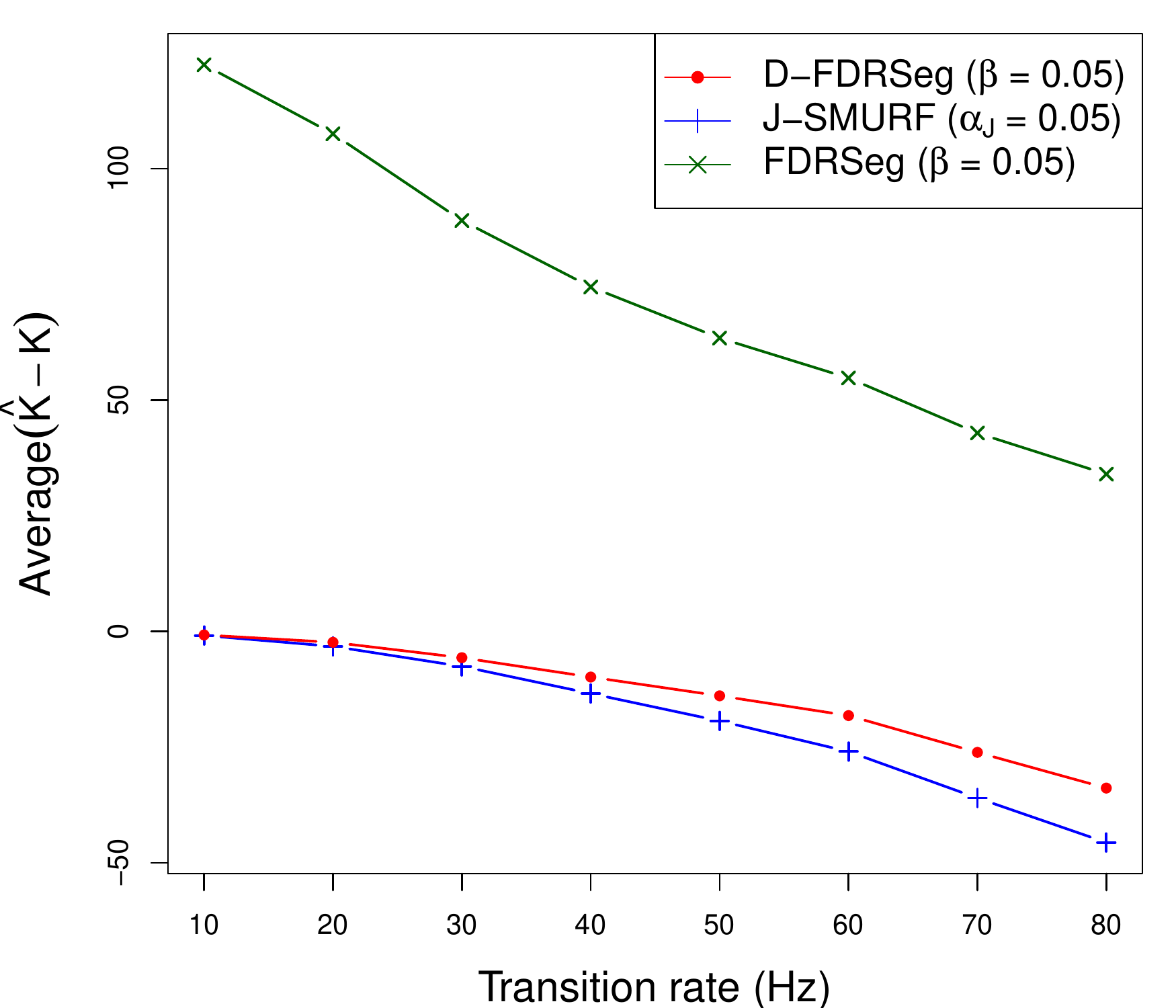}  
\caption{{Simulation study on a two-state Markov chain with different transition rates ($\text{SNR} = 3$).} }  
 \label{pic:ion_channel:sim:study}
\end{figure}

\begin{figure}[!h]
\centering
\includegraphics[width= 1\columnwidth]{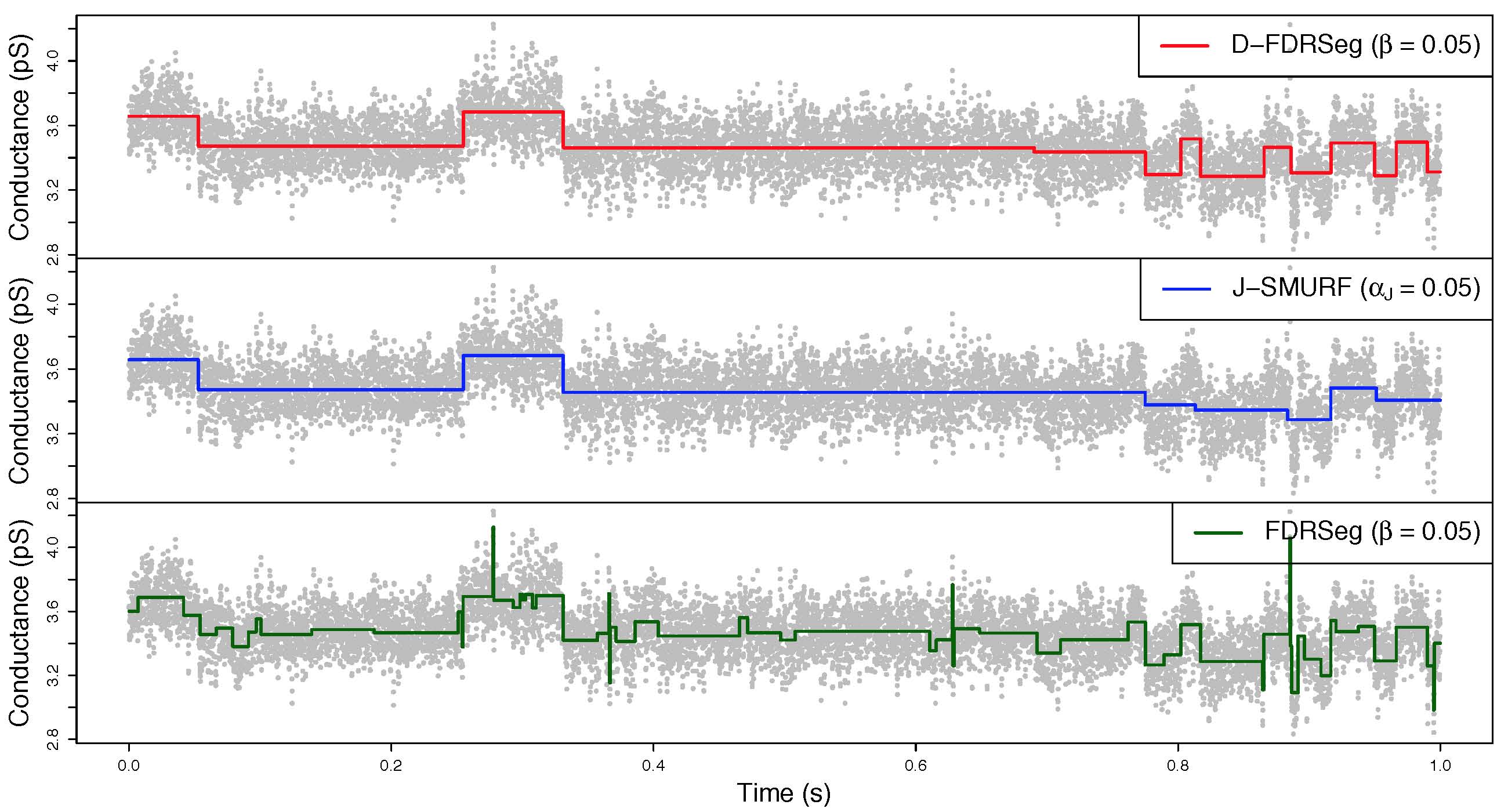} 
\caption{{The time trace of conductance for gramicidin A. } }  
 \label{pic:ion_channel:dfdr:jsmurf}
\end{figure}

\section{Conclusion and discussion}\label{sec:conclusion}

In this work we proposed a multiple change-point {segmentation method} FDRSeg, which is based on the relaxation of FWER to FDR.  By experiments on both simulation and real data, FDRSeg shows high detection power with controlled accuracy. 
A theoretical bound is provided for its FDR, which provides a meaningful interpretation of the only user-specified parameter $\alpha$. In addition, {we have shown that} jump locations are detected at the optimal sampling rate $1/n$ up to a log-factor. Concerning the signal, i.e. both jump locations and function values, the convergence rate of the estimator is minimax optimal {w.r.t. $L^p$-risk ($p \ge 1$)} up to a log-factor. {This result is over classes of step signals with bounded jump sizes, and either bounded, or possibly increasing, number of jumps. } 

Our method is not confined to i.i.d. Gaussian observations, although we restricted our presentation to this in order to highlight the main ideas more concisely. Obviously, it can be extended to more general additive errors, because the proof of Lemma~\ref{lem:FD_noise} only relies on Gaussianity for the independence of the residuals and the mean. In the case of different models, e.g. exponential family regression, we believe that one can argue {along} similar lines as in the proof of Theorem~\ref{thm:mainthm}, but results will only hold asymptotically. This, however, is above the scope of the paper, and postponed to {further} research. 
Also, as we have applied the CBS outlier smoothing procedure to the array CGH data, it  might be of interest to have more robust versions of FDRSeg. To this end, e.g. local median, instead of local mean, might provide useful results. Alternatively, one may transform this into a Bernoulli regression problem~(see~\citep{DueKov09, FriMunSie14}), which might be interesting for {further} research. 
In the paper, we also suggested a modification of FDRSeg for dependent data, which shows attractive empirical results. It would be of interest to study this modified estimator from {a} theoretical point of view {as well}.  

\textbf{Acknowledgement.}
The authors thank Florian Pein, and Inder Tecuapetla for helpful discussions, 
and the Steinem lab (Institute of Organic and Biomolecular Chemistry, University of G{\"o}ttingen) for providing the ion channel data. 

\appendix
\section{Technical proofs}
\subsection{Proof of Theorem~\ref{thm:mainthm}}\label{sec:proof_main_theorem}
The proof of Theorem~\ref{thm:mainthm} relies on two lemmata. As a convention, all results are concerning  the FDRSeg $\hat\mu$ in~\eqref{def:fdr_smuce} without explicit statement. The first one gives a bound for the expected number of false discoveries (FD) given no true discoveries ($\text{TD}=0$), see Section~\ref{sec:intro} for the definitions. 
\begin{lem}\label{lem:FD_noise}
Under above notations, we have for $0 < \alpha < 1/3$
\begin{equation*}
 \E{\mathrm{FD}(\alpha) | \text{TD}(\alpha)=0} \leq \frac{2 \alpha}{1-3 \alpha} =: G(\alpha).
\end{equation*}
\end{lem}

\begin{proof}
  Note that it suffices to prove the result for a constant signal, which we assume w.l.o.g. to be constant zero.
  The proof is then based on the following observation. Assume there exists an estimate $\tilde\mu = \sum_{k=0}^{\tilde K} c_k \mathbf{1}_{\tilde I_k}$ with $(\tilde K+1)$ segments 
  $\biguplus_{k=0}^{\tilde K} \tilde I_k = [0,1),$ 
  which fulfills the multiscale side-constraint $\mathcal{C}_{\tilde K}$ in~\eqref{side-constraint:fdrsmuce}.
Then, the FD of FDRSeg is bounded by $\tilde K$, since it minimizes the number of change-points $k$ among all nonempty $\mathcal{C}_{k}$'s. We will prove the result by constructing such an estimate $\tilde \mu$ and show that $\mathbf{E}[{\tilde K}]\leq {2 \alpha}/{(1-3 \alpha)}$. The estimate $\tilde\mu$ is given by an iterative rule to include change-points until the multiscale side-constraint $\mathcal{C}_{\tilde K}$ is fulfilled.

We first check the whole interval $[0,1)$ whether its mean value $\bar Y$ satisfies the multiscale side-constraint. 
If $T_{[0,1)}(Y, \bar{Y}) \le q_{\alpha}(n)$, then $\tilde\mu := \bar{Y}\mathbf{1}_{[0,1)}$. 
Otherwise, we randomly choose  $i^*$ and $j^*$ from 
\begin{equation}\label{def:istar_jstar}
\left\{ (i,j) : t_{[0,1)}\left(\left[\frac{i}{n}, \frac{j}{n}\right]\right):=\frac{\abs{\sum_{l=i}^{j}(Y_l - \bar Y)}}{\sigma\sqrt{j-i+1}}-\text{pen}\left(\frac{j-i+1}{n}\right)  - q_{\alpha}(n) > 0\right\},
\end{equation}
according to any distribution which is independent of the values $t_{[0,1)}([{i}/{n}, {j}/{n}])$'s.
Then we check intervals $[0,i^*/n)$, $[i^*/n, j^*/n]$ and $(j^*/n, 1)$ individually, and split them further in the same manner if necessary. This procedure is repeated until on each resulting interval $I$ its mean value $\bar Y_I$ satisfies the multiscale side-constraint, i.e. $T_I(Y,\bar Y_I) \le q_{\alpha}(\#{I})$.  Finally, $\tilde \mu := \sum_{I} \bar Y_I \mathbf{1}_{I}$.

Let $D_k$ denote the number of change-points (discoveries) and $S_k$ the number of segments introduced in the $k$-th step. We make the convention that $D_k = S_k = 0$ if the procedure stops before the $k$-th step. 
It follows from $\P\left\{T_{[0,1)}(Y,\bar Y) > q_{\alpha}(n)\right\} \le \alpha$, cf.~\eqref{def:qalpha}, (recall $Y_i = \varepsilon_i$ here) that
\[
\E{D_1} \leq 2 \alpha \mbox{ and } \E{S_1} \le 3 \alpha.
\]
Now we consider the three intervals $I_1=[0,i^*/n),\, I_2=[i^*/n,j^*/n]$ and $I_3=(j^*/n,1)$ and bound the probability of further splitting them into smaller intervals.  It will be shown that 
\begin{equation*}
\P\left\{T_{I_k}(Y, \bar Y_{I_k}) > q_{\alpha}(\#{I_k})\, \middle\vert\, T_{[0,1)}(Y, \bar{Y}) > q_{\alpha}(n) \right\} \le \alpha \quad \text{ for } k = 1, \, 2, \, 3.
\end{equation*}
Given $I_2 =[i/n, j/n]$, the random variable $T_{I_k}(Y, \bar Y_{I_k})$ depends only on $\{Y_{i} - \bar Y_{I_k}$, $i/n \in I_k\}$, which is independent of $\bar Y$ and $\bar Y_{I_2}$. It follows from~\eqref{def:istar_jstar} that $t_{[0,1)}(I_2)$ depends only on $\bar Y$ and $\bar Y_{I_2}$. Thus $T_{I_k}(Y, \bar Y_{I_k})$ is independent of $t_{[0,1)}(I_2)$ conditioned on $I_2$.
\begin{align*}
&\P\left\{T_{I_k}(Y, \bar Y_{I_k}) > q_{\alpha}(\#{I_k})\, \middle\vert\, T_{[0,1)}(Y, \bar{Y}) > q_{\alpha}(n) \right\}\\
=&\sum_{0\le i \le j < n}\P\left\{T_{I_k}(Y, \bar Y_{I_k}) > q_{\alpha}(\#{I_k})\, \middle\vert\, t_{[0,1)}(I_2) > 0, I_2=\left[\frac{i}{n},\frac{j}{n}\right]\right\}\\
&\qquad\qquad\qquad\qquad\times\P\left\{I_2=\left[\frac{i}{n},\frac{j}{n}\right]\, \middle\vert\, T_{[0,1)}(Y, \bar{Y}) > q_{\alpha}(n) \right\}\\
=&\sum_{0\le i \le j < n}\P\left\{T_{I_k}(Y, \bar Y_{I_k}) > q_{\alpha}(\#{I_k})\, \middle\vert\, I_2=\left[\frac{i}{n},\frac{j}{n}\right]\right\}\\
&\qquad\qquad\qquad\qquad\times\P\left\{I_2=\left[\frac{i}{n}, \frac{j}{n}\right]\, \middle\vert\, T_{[0,1)}(Y, \bar{Y}) > q_{\alpha}(n) \right\}\\
\le &  \sum_{0\le i \le j < n}\alpha\P\left\{I_2=\left[\frac{i}{n}, \frac{j}{n}\right]\, \middle\vert\, T_{[0,1)}(Y, \bar{Y}) > q_{\alpha}(n) \right\} \le \alpha.
\end{align*}
It follows that
$$
\E{D_2|S_{1}} \leq 2 \alpha S_1 \mbox{ and } \E{S_2|S_1} \le 3 \alpha S_1
$$
Using the same line of argumentation we find in general that 
$$
\E{D_k|S_{k-1}} \leq 2 \alpha S_{k-1}\mbox{ and }\E{S_k|S_{k-1}} \leq 3 \alpha S_{k-1}.
$$ 
It implies 
\begin{align*}
\E{D_k} &=\E{\E{D_k|S_{k-1}}} \le 2 \alpha \E{S_{k-1}} = 2 \alpha \E{S_{k-1}|S_{k-2}} \\
&\le 2 \alpha  \cdot 3 \alpha \E{S_{k-2}} \le 2\alpha (3\alpha)^{k-1}.
\end{align*}
Hence,
\begin{equation*}
\E{\text{FD}} \le \E{\tilde K} = \E{\sum_{k=1}^\infty D_k}  = \sum_{k=1}^\infty \E{D_k}  
 \le \sum_{k=1}^\infty 2 \alpha \left(3\alpha \right)^{k-1} = \frac{2 \alpha}{1-3\alpha}.
\end{equation*}
\end{proof}

The next lemma shows the expected FD conditioned on TD.
\begin{lem}\label{lem:rel_FD_TD}
$\E{\text{FD}(\alpha)| \text{TD}(\alpha)=\kappa} \leq (\kappa+1) \E{\text{FD}(\alpha)} \leq  (\kappa+1) G(\alpha)$.
\end{lem}

\begin{proof}
\begin{equation*}
\begin{aligned}
&\E{\text{FD}\,\middle\vert\, \text{TD} = \kappa}\\
 = &\sum_{i_1< \cdots < i_\kappa} \E{\text{FD}\,\middle\vert\, \hat \tau_{i_1}, \ldots, \hat \tau_{i_\kappa} \text{ are true},  \text{TD} = \kappa}\P\left\{\hat \tau_{i_1}, \ldots, \hat \tau_{i_\kappa} \text{ are true}\,\middle\vert\,\text{TD} = \kappa\right\}\\
 =& \sum_{i_1< \cdots < i_\kappa} \sum_{j=0}^{\kappa}\E{\text{FD}|_{(\hat \tau_{i_j},\hat\tau_{i_{j+1}})}\,\middle\vert\,\hat \tau_{i_1}, \ldots, \hat \tau_{i_\kappa} \text{ are true},  \text{TD} = \kappa}\\
 &\qquad\qquad\qquad\qquad\times\P\left\{\hat \tau_{i_1}, \ldots, \hat \tau_{i_\kappa} \text{ are true}\,\middle\vert\,\text{TD} = \kappa\right\},
\end{aligned}
\end{equation*}
where $\tau_{i_{0}} := 0$ and $\tau_{i_{\kappa+1}} := 1$. Note that there is no true discovery on $(\hat \tau_{i_j},\hat\tau_{i_{j+1}})$, $j = 0,\ldots,\kappa$. By applying Lemma~\ref{lem:FD_noise} to each segment on  $(\hat \tau_{i_j},\hat\tau_{i_{j+1}})$, we have
\begin{align*}
\E{\text{FD}\,\middle\vert\, \text{TD} = \kappa}
&\le\sum_{i_1< \cdots < i_\kappa} \sum_{j=0}^{\kappa}G(\alpha) \P\left\{\hat \tau_{i_1}, \ldots, \hat \tau_{i_\kappa} \text{ are true}\,\middle\vert\,\text{TD} = \kappa\right\} \\
&\le (\kappa+1)G(\alpha).
\end{align*}
\end{proof}

Now we are ready to prove Theorem~\ref{thm:mainthm}.

\begin{proof}[Proof of Theorem~\ref{thm:mainthm}]
For random variables $X$, $Y$ and $Z=X+Y$ we find by Jensen's inequality that
  \begin{align*}
    \E{\E{\frac{X}{Z}\middle| Y}} \leq \E{\frac{\E{X\middle|Y}}{Y+\E{X\middle|Y}}}.
  \end{align*}
  We set $X=\text{FD}$, $Y= \text{TD}+1$. Together with Lemma \ref{lem:rel_FD_TD} this yields that
  \begin{align*}
  \text{FDR} =  \E{\frac{X}{Z}} =  \E{\E{\frac{X}{Z}\middle| Y}} \leq  \E{\frac{\E{X\middle|Y}}{Y+\E{X\middle|Y}}} \leq \frac{G(\alpha)}{1+G(\alpha)} = \frac{2\alpha}{1-\alpha}.
  \end{align*}
\end{proof}

\subsection{Proof of Theorem~\ref{thm:accuracy:location}}\label{sec:proof_location}
{
\begin{lem}[Upper bound for quantiles]\label{lem:bnd:quantile}
Let $q_{\alpha}(n)$ be given in~\eqref{def:qalpha}. Then there is a constant $C$ such that
\[
\sup_{n \ge 1} q_{\alpha}(n) \le C + \sqrt{2 \log \frac{1}{\alpha}} \quad \text{ for all } \alpha \in (0,1).  
\]
\end{lem}
\begin{proof}
Let $a_{ij} \in \R^n$, $0\le i \le j < n$, be given by
\[
a_{ij} := 
\begin{cases}
\frac{1}{\sqrt{j-i+1}} & \text{ if } i \le k \le j \\
0 & \text{otherwise}
\end{cases},
\]
and $A := \{a_{ij}: 0 \le i \le j < n\}$. Let also $\xi \sim \mathcal{N}(0, I_n)$ with $I_n$ the $n$-dimensional identity matrix, and $\mathbf{1} := (1, \ldots, 1)^t \in \R^n$. Then $q_{\alpha}(n)$ is the upper $\alpha$-quantile of $T_n$,
\[
T_n := \max_{a \in A \cup (-A)} a^t (I_n - \frac{1}{n}\mathbf{1}\mathbf{1}^t)\xi - \lambda_a,
\]
where $\lambda_a = \lambda_{-a} = \sqrt{2\log(en\norm{a}_{\infty}^2)}$. Define $f:\R^n \to \R$ by 
\[
f(x) = \max_{a \in A \cup (-A)} a^t (I_n - \frac{1}{n}\mathbf{1}\mathbf{1}^t)x - \lambda_a \quad \text{ for } x\in\R^n.
\] 
It follows that for $x_1,x_2\in\R^n$,
\begin{align*}
\abs{f(x_1)-f(x_2)}& \le \max_{a\in A\cup(-A)} \abs{a^t(I_n - \frac{1}{n}\mathbf{1}\mathbf{1}^t)(x_1-x_2)}  \\
&\le  \max_{a\in A\cup(-A)} \norm{(I_n - \frac{1}{n}\mathbf{1}\mathbf{1}^t)a}\norm{x_1-x_2} \\
&\le \max_{a\in A\cup(-A)}\norm{a}\norm{x-y} \le \norm{x_1-x_2}.
\end{align*}
That is, $f$ is Lipschitz continuous with constant 1. By~\citep[Lemma A.2.2]{VaaWel96} we have
\begin{equation}\label{eq:exp:bnd:mrtail}
\P\left\{T_n - \E{T_n} > t \right\}=\P\left\{f(\xi) - \E{f(\xi)} > t \right\} \le e^{-\frac{t^2}{2}} \quad \text{ for } t \ge 0. 
\end{equation}
It follows from~\citep{Vitale00} that 
\[
\E{T_n} \le \E{\max_{a \in A \cup (-A)} a^t\xi - \lambda_a}.
\]
By~\citep{FriMunSie14} we further have
\[
\E{\max_{a \in A \cup (-A)} a^t\xi - \lambda_a}\le \E{\sup_{0\le s< t\le 1} \frac{\abs{B(t)-B(s)}}{\sqrt{t-s}}-\sqrt{\log\frac{e}{t-s}}} := C < \infty
\]
where $B(t)$ is a standard Brownian motion. It together with~\eqref{eq:exp:bnd:mrtail} implies
\[
q_{\alpha}(n) \le \E{T_n} + \sqrt{2\log\frac{1}{\alpha}} \le C + \sqrt{2\log \frac{1}{\alpha}}\quad \text{ for all } n\in\N.
\] 
\end{proof}
}

\begin{proof}[Proof of Theorem~\ref{thm:accuracy:location}]
(i) This follows from the proof of Theorem 7 in~\citep{FriMunSie14} by replacing $q$ by $\max_{m \le n} q_{\alpha}(m)$. \newline
(ii) Given any $\epsilon > 0$, it follows by Lemma~\ref{lem:bnd:quantile} that 
\[
\max_{m \le n} q_{\alpha_n}(m) \le C + \sqrt{2\log \frac{1}{\alpha_n}}  \le \sqrt{2(1+\epsilon)\gamma\log n}, 
\]
for sufficiently large $n$. Then, elementary calculation and~\eqref{eq:bound:location:accuracy} shows the assertion.  
\end{proof}

\subsection{Proof of Theorem~\ref{thm:convergence:rate:LpRisk}}\label{sec:proof_convergence_rate}

Let $\hat{K}_n$ be the number of change-points of FDRSeg $\hat\mu_{n, \alpha_n}$, and $q_n:=\max_{m \le n} q_{\alpha_n}(m)$. The control of FDR implies a bound on overestimation of the number of change-points. 
\begin{lem}[Overestimation bound]\label{lem:overestimation:bound}
$$\P\{\hat K_n > K\} \le (K+2)\frac{2\alpha_n}{1-\alpha_n}.$$
\end{lem}
\begin{proof}
\begin{align*}
\P\{\hat K_n > K\}& = \P\{(\hat K_n -K)_+ \ge 1\} \le \P\left\{\frac{(\hat K_n - K)_+}{(\hat K_n-K)_+ + K +1} \ge\frac{1}{1 + K + 1}\right\} \\
& \le (K+2)\E{\frac{(\hat K_n - K)_+}{\hat K_n + 1}} \le (K+2)\E{\frac{\text{FD}}{\hat K_n + 1}} \le (K+2)\frac{2\alpha_n}{1-\alpha_n},
\end{align*}
where the last inequality follows from Theorem~\ref{thm:mainthm}.
\end{proof}

\begin{proof}[Proof of Theorem~\ref{thm:convergence:rate:LpRisk} (i)]
Let $p_* := 1/\min\{1/2, 1/p\}$. Note that 
\begin{equation*}
\E{\norm{\hat{\mu}_{n,\alpha_n} - \mu}_{L^p}} =  \int_0^{\sqrt{n}}\P\left\{\norm{\hat\mu_{n,\alpha_n}-\mu}_{L^p} \ge s\right\}ds + \int_{\sqrt{n}}^{\infty}\P\left\{\norm{\hat\mu_{n,\alpha_n}-\mu}_{L^p} \ge s\right\}ds. 
\end{equation*}
In the following, we will show as $n \to \infty$,
\begin{eqnarray}
\sup_{\mu \in B_{\nu,\epsilon,L}}\int_0^{\sqrt{n}}\P\left\{\norm{\hat\mu_{n,\alpha_n}-\mu}_{L^p} \ge s\right\}ds\left(\frac{\nu\epsilon^2 n}{\sigma^2\log n}\right)^{1/p_*} &\le & 25 L,  \label{eq:good:noise}\\
\sup_{\mu \in B_{\nu,\epsilon,L}}\int_{\sqrt{n}}^{\infty}\P\left\{\norm{\hat\mu_{n,\alpha_n}-\mu}_{L^p} \ge s\right\}ds\left(\frac{\nu\epsilon^2 n}{\sigma^2\log n}\right)^{1/p_*} & \to & 0. \label{eq:bad:noise}
\end{eqnarray}
Then, the assertion of the theorem holds by combining~\eqref{eq:good:noise} and~\eqref{eq:bad:noise}. 
\newline Verification of~\eqref{eq:good:noise}: Let us choose 
\begin{equation}\label{eq:choice:delta:n}
\delta_n := 129\frac{\sigma^2\log n}{\epsilon^2 n}. 
\end{equation}
Note that 
\begin{equation}\label{eq:split:three:parts}
\begin{split}
& \int_0^{\sqrt{n}}\P\left\{\norm{\hat\mu_{n,\alpha_n}-\mu}_{L^p} \ge s\right\}ds\left(\frac{\nu\epsilon^2 n}{\sigma^2\log n}\right)^{1/p_*} \\
\le &  \sqrt{n}\left(\frac{\nu\epsilon^2 n}{\sigma^2\log n}\right)^{1/p_*} \left(\P\{\hat K_n > K_{\mu}\}  +  \P\left\{d\left(\mu,\hat\mu_{n, \alpha_n}\right) > \delta_n\right\} \vphantom{\int_0^{\infty}\P\left\{\norm{\hat \mu_{n, \alpha_n}-\mu}_{L^p} \ge s; \hat K_n \le K_{\mu}; d(\mu,\hat\mu_{n, \alpha_n}) \le \delta_n\right\}ds}\right. \\
&\quad +\left. \int_0^{\infty}\P\left\{\norm{\hat \mu_{n, \alpha_n}-\mu}_{L^p} \ge s; \hat K_n \le K_{\mu}; d(\mu,\hat\mu_{n, \alpha_n}) \le \delta_n\right\}ds\right). 
\end{split}
\end{equation}
For the first term in~\eqref{eq:split:three:parts}, it follows from Lemma~\ref{lem:overestimation:bound} that 
\begin{align*}
& \mathop{\lim\sup}_{n\to\infty} \sup_{\mu \in B_{\nu, \epsilon, L}} \sqrt{n}\left(\frac{\nu\epsilon^2 n}{\sigma^2\log n}\right)^{1/p_*} \P\{\hat K_n > K_{\mu}\} \\
\le & \mathop{\lim\sup}_{n\to\infty} \sup_{\mu \in B_{\nu, \epsilon, L}} \sqrt{n}\left(\frac{\nu\epsilon^2 n}{\sigma^2\log n}\right)^{1/p_*} (K_{\mu} +2)\frac{2\alpha_n}{1-\alpha_n} \\
\le & \mathop{\lim\sup}_{n\to\infty} \sqrt{n}\left(\frac{\nu\epsilon^2 n}{\sigma^2\log n}\right)^{1/p_*} \left(\frac{1}{\nu} +2\right)\frac{2\alpha_n}{1-\alpha_n} = 0. 
\end{align*}
For the second term in~\eqref{eq:split:three:parts}, by elementary calculation, one can derive from Theorem~\ref{thm:accuracy:location} and~\eqref{eq:choice:delta:n}  that 
\[
\limsup_{n \to \infty} \sup_{\mu \in B_{\nu, \epsilon, L}}\sqrt{n}\left(\frac{\nu\epsilon^2 n}{\sigma^2\log n}\right)^{1/p_*} \P\left\{d\left(\mu,\hat\mu_{n, \alpha_n}\right) > \delta_n\right\} = 0. 
\]
Now we consider the last term in~\eqref{eq:split:three:parts}. Let $ \{\tau_i; i = 1,\ldots,K_{\mu}\} $ be the change-points of $\mu$, $\{\hat\tau_i; i = 1,\ldots,\hat K_n\}$ the change-points of $\hat \mu_{n, \alpha_n}$. Both are ordered increasingly. By~\eqref{eq:choice:delta:n}, we have $\delta_n < \nu/2$ for large enough $n$. It implies that $\hat K_n = K_{\mu}$ and $\abs{\tau_i - \hat{\tau}_i} \le \delta_n$ for $i = 1, \ldots, K_{\mu}$. Let $\norm{f}_{I, \infty} := \max_{x \in I}\abs{f(x)}$, 
\begin{align*}
\mathcal{I}_n &:= \left\{[0, \tau_1-\delta_n), (\tau_1+\delta_n, \tau_2-\delta_n), \ldots, (\tau_{K_{\mu}}+\delta_n, 1)\right\},\\ 
\text{and }\mathcal{J}_n& := \{[\tau_i-\delta_n, \tau_i+\delta_n]; i = 1,\ldots,K_{\mu}\}. 
\end{align*}
For $I \in \mathcal{I}_n$, we have 
$$
\hat{\mu}_{n, \alpha_n} \equiv \hat{\mu}_{I,n},\text{ and }\mu \equiv \mu_{I}\quad\text{ on } I, 
$$ 
for some constants $\hat{\mu}_{I,n}$ and $\mu_{I}$. Note that 
$$
\sqrt{n\abs{I}}\frac{\abs{\bar{Y}_I - \hat\mu_{I,n}}}{\sigma} \le q_n + \sqrt{2\log (\frac{e}{\abs{I}})},
$$ 
which implies $\sqrt{n\abs{I}}\abs{\bar{Y}_I - \mu_I-\sigma s}/\sigma \le q_n + \sqrt{2\log (e/\abs{I})}$, if $\bar{Y}_I -\mu_I \le \sigma s$ and $\hat\mu_{I,n} -\mu_I > \sigma s$. Then, 
\begin{align*}
\P\{\hat\mu_{I,n} - \mu_I \ge \sigma s\} &\le \P\{\bar{Y}_I -\mu_I \le \sigma s; \hat\mu_{I,n} -\mu_I > \sigma s\} + \P\{\bar{Y}_I > \mu_I + \sigma s\} \\
& \le \P\left\{\sqrt{n\abs{I}}\abs{\frac{\bar{Y}_I - \mu_I}{\sigma}-s} \le  q_n + \sqrt{2\log \frac{e}{\abs{I}}} \right\} + \P\{\bar{Y}_I > \mu_I + \sigma s\} \\
& \le \exp\left(-\frac{1}{8}\left(s\sqrt{n\abs{I}} - q_n -  \sqrt{2\log \frac{e}{\abs{I}}}\right)_+^2\right) + \exp\left(-\frac{n\abs{I} s^2}{2}\right) \\
& \le 2\exp\left(-\frac{1}{8}\left(s\sqrt{n\abs{I}} - q_n - \sqrt{2\log \frac{e}{\abs{I}}}\right)_+^2\right), \\
& \le 2\exp\left(-\frac{1}{8}\left(s\sqrt{n(\lambda_{\mu}-2\delta_n)} - q_n - \sqrt{2\log \frac{e}{\lambda_{\mu}-2\delta_n}}\right)_+^2\right).
\end{align*}
The third inequality above follows from Lemmata 7.1 and 7.3 in~\citep{FriMunSie14}. By the symmetry of {the} Gaussian distribution, the same bound can be shown for $\P\{\hat\mu_{I,n} - \mu_I \le -\sigma s\}$. Thus, for each $I \in \mathcal{I}_n$, 
\[
\P\{\abs{\hat\mu_{I,n} - \mu_I} \ge \sigma s\} \le 4\exp\left(-\frac{1}{8}\left(s\sqrt{n(\lambda_{\mu}-2\delta_n)} - q_n - \sqrt{2\log \frac{e}{\lambda_{\mu}-2\delta_n}}\right)_+^2\right).
\]
For each $J \in \mathcal{J}_n$, we have 
$$
\norm{\hat{\mu}_{n, \alpha_n} -\mu}_{J, \infty} \le \max_{I\in \mathcal{I}_n}\abs{\hat{\mu}_{I,n} -\mu_I} + \tilde{\Delta}_{\mu} \le \max_{I\in \mathcal{I}_n}\abs{\hat{\mu}_{I,n} -\mu_I} + L. 
$$
Therefore, 
\begin{align*}
    &\P\left\{\norm{\hat \mu_{n, \alpha_n}-\mu}_{L^p} \ge s;\, \hat K_n = K_{\mu};\, \abs{\tau_i - \hat{\tau}_i} \le \delta_n\text{ for } i = 1,\ldots, K_{\mu}\right\} \nonumber\\
\le & \P\left\{\sum_{I \in \mathcal{I}_n}\abs{I}\abs{\hat\mu_{I,n}-\mu_I}^p + \sum_{J\in\mathcal{J}_n}\abs{J}\norm{\hat{\mu}_{n, \alpha_n}-\mu}^p_{J,\infty}\ge s^p\right\}\nonumber \\
\le &\P\left\{\max_{I\in\mathcal{I}_n}\abs{\hat\mu_{I,n}-\mu}^p\sum_{I \in \mathcal{I}_n}\abs{I}+\left(\max_{I\in \mathcal{I}_n}\abs{\hat{\mu}_{I,n} -\mu_I} + L\right)^p\sum_{J \in \mathcal{J}_n}\abs{J}\ge s^p\right\} \nonumber\\
\le &\P\left\{\max_{I\in\mathcal{I}_n}\abs{\hat\mu_{I,n}-\mu}^p\sum_{I \in \mathcal{I}_n}\abs{I}+\left(2^{p-1}\max_{I\in \mathcal{I}_n}\abs{\hat{\mu}_{I,n} -\mu_I}^p + 2^{p-1}L^p\right)\sum_{J \in \mathcal{J}_n}\abs{J}\ge s^p\right\} \nonumber\\
\le & \P\left\{\max_{I\in\mathcal{I}_n}\abs{\hat\mu_{I,n}-\mu}^p(1-2 K_{\mu} \delta_n+2^p K_{\mu} \delta_n ) \ge s^p - 2^pL^pK_{\mu}\delta_n\right\}\nonumber \\
\le & \sum_{I\in\mathcal{I}_n}\P\left\{\abs{\hat\mu_{I,n}-\mu}^p(1-2 K_{\mu} \delta_n+2^p K_{\mu} \delta_n ) \ge s^p - 2^pL^p K_{\mu}\delta_n\right\}\nonumber \\
\le &4 (K_{\mu}+1)\exp\left(-\frac{1}{8}\left(\frac{\sqrt{n(\lambda_{\mu}-2\delta_n)}}{\sigma}\left(\frac{s^p-2^pL^p K_{\mu}\delta_n}{1-2 K_{\mu} \delta_n +2^p K_{\mu} \delta_n }\right)^{1/p}\vphantom{-q_n - \sqrt{2\log\frac{e}{\lambda_{\mu} - 2\delta_n}}}\right.\right. \nonumber \\
& \hphantom{4 (K_{\mu}+1)\exp}\quad\left.\vphantom{-\frac{1}{8}}\left.\vphantom{\frac{\sqrt{n(\lambda_{\mu}-2\delta_n)}}{\sigma}\left(\frac{s^p-2^pL^p K_{\mu}\delta_n}{1-2 K_{\mu} \delta_n +2^p K_{\mu} \delta_n }\right)^{1/p} }-q_n - \sqrt{2\log\frac{e}{\lambda_{\mu} - 2\delta_n}}\right)_+^2\right).\label{eq:good:term}
\end{align*}
Let us choose 
$$
s_* := 25 L \left(\frac{\sigma^2\log n}{\nu \epsilon^2 n}\right)^{1/p_*}.
$$ 
Then, for large enough $n$, we have
\begin{align*}
& \left(\frac{\nu\epsilon^2 n}{\sigma^2\log n}\right)^{1/p_*} \int_0^{\infty}\P\left\{\norm{\hat \mu_{n, \alpha_n}-\mu}_{L^p} \ge s; \hat K_n \le K_{\mu}; d(\mu,\hat\mu_{n, \alpha_n}) \le \delta_n\right\}ds \\
\le& 25 L +  (\frac{4}{\nu}+4)L\int_{25}^{\infty}\exp\left(-\frac{\log n}{8}\left(\frac{L}{\sqrt{2}\epsilon}(s-2\cdot128^{1/p_*})-2\right)_+^2\right)ds \\
= & 25 L + (\frac{8}{\nu}+8) \epsilon \sqrt{\frac{\pi}{\log n}} \to 25 L, \text{ uniformly over } B_{\nu, \epsilon, L}, \text{ as }n \to \infty. 
\end{align*}
Thus, we have shown \eqref{eq:good:noise}.
\newline Verification of~\eqref{eq:bad:noise}: From $\hat\mu_{n,\alpha_n} \in \mathcal{C}_{\hat K_n}$ it follows
\[
\norm{\hat\mu_{n,\alpha_n} - \sum_{i=0}^{n-1} Y_i \mathbf{1}_{[\frac{i}{n},\frac{i+1}{n})}}_{L^p} \le \max_{0 \le i \le n-1}\abs{\hat\mu_{n,\alpha_n}(\frac{i}{n}) -Y_i} \le C\sigma\sqrt{\log n},
\]
for some constant $C$. 
Denoting $\mu = \sum_{k=0}^{K_{\mu}} c_k \mathbf{1}_{[\tau_k,\tau_{k+1})}$, we have
\begin{align*}
\norm{\sum_{i=0}^{n-1} Y_i \mathbf{1}_{[\frac{i}{n},\frac{i+1}{n})} - \mu}_{L^p} \le & \norm{\sum_{i=0}^{n-1} Y_i \mathbf{1}_{[\frac{i}{n},\frac{i+1}{n})} - \sum_{k=0}^{K_{\mu}} c_k \mathbf{1}_{[\frac{\lceil n \tau_k\rceil}{n},\frac{\lceil n \tau_{k+1}\rceil}{n})}}_{L^p} \\
& \qquad + \norm{\sum_{k=0}^{K_{\mu}} c_k \mathbf{1}_{[\frac{\lceil n \tau_k \rceil}{n},\frac{\lceil n \tau_{k+1} \rceil}{n})} - \mu}_{L^p} \\
\le & \sigma\left(\frac{1}{n}\sum_{i = 0}^{n-1}\abs{\varepsilon_i}^p\right)^{1/p} + \tilde{\Delta}_{\mu}\left(\frac{{K_{\mu}}}{n}\right)^{1/p} \\
\le & \sigma\left(\frac{1}{n}\sum_{i = 0}^{n-1}\abs{\varepsilon_i}^p\right)^{1/p} + L\left(\frac{1}{n\nu}\right)^{1/p}. 
\end{align*}
If $n$ is large enough such that $\sqrt{n}/2 \ge L/(n\nu)^{1/p} + C\sigma\sqrt{\log n}$, then
\begin{align*}
&\int_{\sqrt{n}}^{\infty} \P\left\{\norm{\hat\mu_{n,\alpha_n} - \mu}_{L^p} \ge s\right\} ds \\
\le &\int_{\sqrt{n}}^{\infty} \P\left\{\norm{\hat\mu_{n,\alpha_n} - \sum_{i=0}^{n-1} Y_i \mathbf{1}_{[\frac{i}{n},\frac{i+1}{n})}}_{L^p} + \norm{\sum_{i=0}^{n-1} Y_i \mathbf{1}_{[\frac{i}{n},\frac{i+1}{n})} - \mu}_{L^p} \ge s\right\} ds \\
\le & \int_{\sqrt{n}}^{\infty}\P\left\{\sigma\left(\frac{1}{n}\sum_{i = 0}^{n-1}\abs{\varepsilon_i}^p\right)^{1/p} + L\left(\frac{1}{n\nu}\right)^{1/p}+C\sigma\sqrt{\log n} \ge s\right\} ds \\
\le & \int_{\sqrt{n}}^{\infty}\P\left\{ \sigma\left(\frac{1}{n}\sum_{i = 0}^{n-1}\abs{\varepsilon_i}^p\right)^{1/p}\ge \frac{s}{2}\right\} ds
\le  \int_{\sqrt{n}}^{\infty}\left(\frac{2\sigma}{s}\right)^{2p} ds\E{\left(\frac{1}{n}\sum_{i = 0}^{n-1}\abs{\varepsilon_i}^p\right)^2} \\
\le & \frac{(2\sigma)^{2p}}{2p-1}\E{\abs{\varepsilon_0}^{2p}}n^{1/2-p}.
\end{align*}
It implies 
\begin{align*}
&\sup_{\mu \in B_{\nu,\epsilon,L}}\int_{\sqrt{n}}^{\infty}\P\left\{\norm{\hat\mu_{n,\alpha_n}-\mu}_{L^p} \ge s\right\}ds\left(\frac{\nu\epsilon^2 n}{\sigma^2\log n}\right)^{1/p_*} \\
\le & \frac{(2\sigma)^{2p}}{2p-1}\E{\abs{\varepsilon_0}^{2p}}n^{1/2-p}\left(\frac{\nu\epsilon^2 n}{\sigma^2\log n}\right)^{1/p_*} \\
\le & (2\sigma)^{2p}\E{\abs{\varepsilon_0}^{2p}}n^{-1/2}\left(\frac{\nu\epsilon^2 n}{\sigma^2\log n}\right)^{1/2} \to 0, \text{ as } n \to \infty.
\end{align*}
\end{proof}
\begin{proof}[Proof of Theorem~\ref{thm:convergence:rate:LpRisk} (ii)]
The proof follows exactly the same way as Theorem~\ref{thm:convergence:rate:LpRisk} (i), if we choose 
\begin{align*}
v:=v_n, \, \delta_n := 175 \frac{\sigma^2\log n}{\epsilon^2 n}, 
\text{ and }s_* := 34 L \left(\frac{\sigma^2\log n}{\nu \epsilon^2 n}\right)^{1/p_*}.
\end{align*}
\end{proof}
\subsection{Proof of Theorem~\ref{thm:convergence:rate:lower:bound}}\label{sec:proof_lower_bound}
By $\chi$ we denote {the observational space of $Y = (Y_0, \ldots, Y_{n-1})$} from model~\eqref{model_1}. 
\begin{lem}[see Section 2.2 in \citep{Tsy09}]\label{lem:lower:bnd:general:approach}
Assume the change-point regression model~\eqref{model_1}, $1 \le p < \infty$, and $B$ a set of step functions. If $\{\mu_1, \ldots, \mu_m\} \subset B$ satisfies
\[
\norm{\mu_i -  \mu_j}_{L^p} \ge 2 s, \quad\text{for } 1 \le i < j \le m,
\] 
then 
\[
\inf_{\hat\mu}\sup_{\mu \in B} s^{-1}\E{\norm{\hat\mu - \mu}_{L^p}} \ge e_m:= \inf_{\psi} \max_{1 \le i \le m} \P_{\mu_i}\{\psi \neq i\},
\]
where the last infimum is taken over all measurable $\psi: \chi \to \{1,\ldots,m\}$.  
\end{lem}
\begin{proof}[Proof of Theorem~\ref{thm:convergence:rate:lower:bound} (i)]
Consider 
\[
\{\mu_1 \equiv 0, \mu_2 \equiv \sigma/\sqrt{n}\} \subset B_{\nu,\epsilon,L}.
\] 
Let $P_i$ be the measure on $\chi$ associated with $\mu_i, i = 0,1$. Then, the Kullback divergence 
\[
K(P_1,P_2) = \frac{n}{2\sigma^2}\left(\frac{\sigma}{\sqrt{n}}\right)^2 = \frac{1}{2}.
\]
By Theorem 2.2 in~\citep{Tsy09}, we have $e_2 \ge 1/4$. Note that 
$$
\norm{\mu_1 - \mu_2}_{L^p} = \frac{\sigma}{\sqrt{n}}. 
$$
It follows from Lemma~\ref{lem:lower:bnd:general:approach} that 
\begin{equation}\label{eq:2:hypothes:bound:a}
\inf_{\hat{\mu}_n}\sup_{\mu \in B_{\nu,\epsilon, L}}\mathbf{E}\left[{\norm{\hat\mu_n-\mu}_{L^p} }\right] \ge \frac{\sigma}{8\sqrt{n}}.
\end{equation}
Consider further 
\[
\{\mu_1 = \mathbf{1}_{[0,\nu)}, \mu_2 = \mathbf{1}_{[0,\nu+\sigma^2/n)}\} \subset B_{\nu,\epsilon,L}.
\]
Similarly, we have 
\begin{align*}
&K(P_1,P_2) \le \frac{n}{2\sigma^2}\frac{\sigma^2}{n} = \frac{1}{2} \implies e_2 \ge \frac{1}{4},\\
\text{and } &\norm{\mu_1 - \mu_2} = \left(\frac{\sigma^2}{n}\right)^{1/p}.
\end{align*}
Then by Lemma~\ref{lem:lower:bnd:general:approach}
\begin{equation}\label{eq:2:hypothes:bound:b}
\inf_{\hat{\mu}_n}\sup_{\mu \in B_{\nu,\epsilon, L}}\mathbf{E}\left[{\norm{\hat\mu_n-\mu}_{L^p} }\right] \ge \frac{1}{8}\left(\frac{\sigma^2}{n}\right)^{1/p}.
\end{equation}
Finally, the assertion follows by~\eqref{eq:2:hypothes:bound:a} and~\eqref{eq:2:hypothes:bound:b}.
\end{proof}
\begin{proof}[Proof of Theorem~\ref{thm:convergence:rate:lower:bound} (ii)]
Consider 
\[
F_{\nu_n}^0:= \left\{\sum_{i=1}^{\lfloor\frac{1}{\nu_n}\rfloor}\frac{(-1)^i+c_i}{2} \mathbf{1}_{[\frac{i-1}{\lfloor 1/ \nu_n\rfloor},\frac{i}{\lfloor 1/ \nu_n\rfloor})}; c_i = \pm\frac{\sigma}{4}\sqrt{\frac{\log 2}{bn^{1-\gamma}}}\right\} \subset B_{\nu_n, \epsilon, L}. 
\]
It is clear that $\#F_{\nu_n}^0 = 2^{\lfloor1/\nu_n\rfloor}$. By Varshamov-Gilbert bound~\citep[Lemma 2.9]{Tsy09}, there is a subset $F_{\nu_n} \subset F_{\nu_n}^0$ such that $\#F_{\nu_n} \ge 2^{\lfloor1/\nu_n\rfloor/8}$ and every two elements in $F_{\nu_n}$ differ on at least $\lfloor1/\nu_n\rfloor/8$ segments. Let $P_i$ be the measure on $\chi$ associated with $\mu_i$, for $\mu_i \in F_{\nu_n}$. 
Then, we {estimate the Kullback divergence as} 
\[
K(P_i,P_j) \le \frac{n}{2\sigma^2}\left(\frac{\sigma}{4}\sqrt{\frac{\log 2}{bn^{1-\gamma}}}\right)^2 \le \frac{\log 2}{32\nu_n} \le \frac{1}{2} \log 2^{\lfloor1/\nu_n\rfloor/8} \le \frac{1}{2}\log(\# F_{\nu_n}).
\]
By Fano's Lemma~\citep[Corollary 2.6]{Tsy09}, it implies that
$e_{\#F_{\nu_n}} \ge 1/4.$ Note that for any $\mu_i, \mu_j \in F_{\nu_n}, i \neq j$, 
\[
\norm{\mu_i - \mu_j}_L^p \ge \frac{\sigma}{4}\sqrt{\frac{\log 2}{bn^{1-\gamma}}}\left(\frac{1}{8}\right)^{1/p} \ge\frac{\sigma}{32}\sqrt{\frac{\log 2}{bn^{1-\gamma}}}. 
\]
It follows from Lemma~\ref{lem:lower:bnd:general:approach} that
\begin{equation}\label{eq:m:hypothes:bound:a}
\inf_{\hat{\mu}_n}\sup_{\mu \in B_{\nu_n,\epsilon, L}}\mathbf{E}\left[{\norm{\hat\mu_n-\mu}_{L^p} }\right] \ge \frac{\sigma}{256}\sqrt{\frac{\log 2}{bn^{1-\gamma}}}.
\end{equation}
Consider further $G_{\nu_n}^0 \subset B_{\nu_n,\epsilon, L}$ given by
\[
G_{\nu_n}^0:=\left\{\sum_{i=1}^{\lfloor\frac{1}{2\nu_n}\rfloor}\frac{(-1)^i}{2}\mathbf{1}_{[\frac{i-1}{\lfloor 1/(2\nu_n)\rfloor}+c_{i-1}, \frac{i}{\lfloor 1/(2\nu_n)\rfloor}+c_{i})}; c_i = \pm\frac{\sigma^2\log2}{16 n}, c_0 = c_{\lfloor 1/(2\nu_n)\rfloor} = 0\right\}. 
\]
Similarly, there is a subset $G_{\nu_n} \subset G^0_{\nu_n}$ such that $\#G_{\nu_n} \ge 2^{(\lfloor1/(2\nu_n)\rfloor-1)/8}$ and every two elements in $G_{\nu_n}$ differ on at least $(\lfloor1/(2\nu_n)\rfloor-1)/8$ change-points.  
Then, we have
\[
K(P_i,P_j) \le \frac{n}{2\sigma^2}\frac{\sigma^2\log2}{8 n}\left(\left\lfloor\frac{1}{2\nu_n}\right\rfloor-1\right) \le \frac{1}{2} \log 2^{(\lfloor1/(2\nu_n)\rfloor-1)/8} \le \frac{1}{2}\log(\# G_{\nu_n}),
\]
which implies that
$e_{\#G_{\nu_n}} \ge 1/4.$ Since, for any $\mu_i, \mu_j \in G_{\nu_n}, i \neq j$, 
\[
\norm{\mu_i - \mu_j}_L^p \ge \left(\frac{\sigma^2\log2}{64 n}\left(\left\lfloor\frac{1}{2\nu_n}\right\rfloor-1\right)\right)^{1/p} \ge \left(\frac{\sigma^2\log2}{256bn^{1-\gamma}}\right)^{1/p} \ge \frac{\log2}{256}\left(\frac{\sigma^2}{bn^{1-\gamma}}\right)^{1/p} ,
\]
then
\begin{equation}\label{eq:m:hypothes:bound:b}
\inf_{\hat{\mu}_n}\sup_{\mu \in B_{\nu_n,\epsilon, L}}\mathbf{E}\left[{\norm{\hat\mu_n-\mu}_{L^p} }\right] \ge \frac{\log2}{2048}\left(\frac{\sigma^2}{bn^{1-\gamma}}\right)^{1/p}.
\end{equation}
Thus, the assertion follows by~\eqref{eq:m:hypothes:bound:a} and~\eqref{eq:m:hypothes:bound:b}. 
\end{proof}

\bibliographystyle{abbrv}

\end{document}